\documentclass[12pt]{amsart}
\usepackage{graphicx}
\usepackage{wrapfig}
\usepackage{times}
\usepackage{amsmath,amssymb,amsthm}
\usepackage{fullpage}
\usepackage[all]{xy}
\usepackage{subfiles}
\usepackage{cite}
\usepackage{color}
\usepackage{url}

\usepackage{mathtools}
\newtagform{EngelLie}[\scriptstyle]{$}{$}
\makeatletter\newcommand{\leqnomode}{\tagsleft@true}
\newcommand{\reqnomode}{\tagsleft@false}\makeatother
\makeatletter
\@addtoreset{subsubsection}{section}
\makeatother

\newtheorem{Theorem}[equation]{Theorem}
\newtheorem{Proposition}[equation]{Proposition}

\newtheorem{Lemma}[equation]{Lemma}

\newtheorem{Corollary}[equation]{Corollary}

\theoremstyle{definition}

\newtheorem{Definition}[equation]{Definition}

\newtheorem{Example}[equation]{Example}

\newcommand{\MM}{\text{\sc m}}
\newcommand{\NN}{\text{\sc n}}

\definecolor{blue}{cmyk}{1.,1.,0.,0.63}
\definecolor{red}{cmyk}{0.,1.,1.,0.63}
\definecolor{green}{cmyk}{1.,0.,1.,0.63}
\definecolor{black}{cmyk}{1.,1.,1.,1.}

\makeatletter
\renewcommand{\@fnsymbol}[1]
{\ensuremath{\ifcase#1\or $*$\or $**$\or $***$\or $****$\or $*****$
\else\@ctrerr\fi}}
\makeatother

\numberwithin{equation}{section}

\newcommand{\style}[1]{\text{\footnotesize{\sf #1}}}

\newcommand{\stylesmall}[1]{{\sf #1}}

\renewcommand{\dim}{\style{dim}}
\newcommand{\dimsmall}{\stylesmall{dim}}

\newcommand{\Dom}{\style{Dom}}

\newcommand{\Ht}{\style{Ht}}

\newcommand{\Jac}{\style{Jac}}
\newcommand{\Jacsmall}{\stylesmall{Jac}}

\renewcommand{\lim}{\style{lim}}

\renewcommand{\max}{\style{max}}

\renewcommand{\min}{\style{min}}
\newcommand{\minsmall}{\stylesmall{min}}

\renewcommand{\mod}{\style{mod}}

\newcommand{\mult}{\style{mult}}

\newcommand{\ord}{\style{ord}}
\newcommand{\ordsmall}{\stylesmall{ord}}

\newcommand{\proj}{\style{proj}}
\newcommand{\projsmall}{\stylesmall{proj}}

\newcommand{\Res}{\style{Res}}

\newcommand{\Spec}{\style{Spec}}

\renewcommand{\sup}{\style{sup}}

\newcommand{\isqrt}{{\scriptstyle{\sqrt{-1}}}}

\newcommand{\vf}{\vfill\end{document}}

\title{Explicit calculation of Siu's Effective Termination
\\
in Kohn's Algorithm for Special Domains in $\mathbb{C}^{3}$}
\author{Wei Guo Foo}
\address{Wei Guo FOO,
D\'{e}partement de Math\'{e}matiques B\^{a}timent 425, 
Facult\'{e} des Sciences d'Orsay Universit\'{e} Paris-Sud, 
F-91405 Orsay Cedex}
\email{wei-guo.foo@math.u-psud.fr}

\begin{document}
\maketitle
\begin{abstract}
In this article, we follow the arguments in a paper of Y-T. Siu to study the effective termination of Kohn's algorithm for special domains in $\mathbb{C}^{3}$. We make explicit the effective constants and generic conditions that appear there, and we obtain an explicit expression for the regularity of the Dolbeault laplacian for the $\overline{\partial}$-Neumann problem. Specifically, on a local 
peudoconvex domain of the special shape
\[
\Omega:=
\bigg\{(z_{1},z_{2},z_{3})\in\mathbb{C}^{3}:\ 
2\text{Re}\ z_{3}+
\sum_{i=1}^{\NN}|F_{i}(z_{1},z_{2})|^{2}<0
\bigg\}
\]
with holomorphic function germs $F_{1},\dots,F_{\NN}\in\mathcal{O}_{\mathbb{C}^{2},0}$ of finite intersection multiplicity
\[
s:=\dimsmall_{\mathbb{C}}\ 
\mathcal{O}_{\mathbb{C}^{2},0}
\big/
\langle 
F_{1},\dots, F_{\NN}
\rangle
<
\infty,
\]
we show that an $\varepsilon$-subelliptic regularity for $(0,1)$-forms holds whenever, just in terms of $s$,
\[
\varepsilon
\geqslant 
\frac{1}{
2^{(4s^{2}-1)s+3}
s^{2}(4s^{2}-1)^{4}
\binom{8s+1}{8s-1}}.
\]
\end{abstract}
\tableofcontents

\section{Introduction}

 Let $(z_{1},z_{2},z_{3})$ be holomorphic coordinates in $\mathbb{C}^{3}$. For some $\NN\geqslant 1$, let $F_{1}(z_{1},z_{2})$, \dots, $F_{\NN}(z_{1},z_{2})$ be holomorphic function germs of $\mathbb{C}^{3}$ vanishing at the origin. We shall study the ideal of subelliptic multipliers on special domains $\Omega\subseteq \mathbb{C}^{3}$ defined by 
\[
\Omega
=
\left\{
2\text{Re}\ z_{3}
+
\sum_{i=1}^{\NN}|F_{i}(z_{1},z_{2})|^{2}<0\right\}.
\]

The idea of subelliptic multipliers was conceived by Joseph J. Kohn in \cite{Kohn-1979} to study the $\overline{\partial}$-Neumann problem on pseudoconvex domains in $\mathbb{C}^{n}$. He constructed what is now known as the {\sl Kohn's algorithm} on subelliptic multipliers to give a geometric interpretation of the subelliptic estimates of the Dolbeault Laplacian. With the use of Diederich--Fornaess' theorem \cite{Diederich-Fornaess-1978}, he proved that the termination of Kohn's algorithm is equivalent to the absence of holomorphic curves passing through the origin in $b\Omega$. 

For special domains in $\mathbb{C}^{2}$, this is equivalent to the fact that in a neighbourhood of the origin, the intersection of the variety germs $\cap_{i=1}^{\NN}\{(z_{1},z_{2}):\ F_{i}(z_{1},z_{2})=0\}$ consists only of the origin. By a result in intersection theory, this means that 
\[
\dim_{\mathbb{C}}\ 
\mathcal{O}_{\mathbb{C}^{2},0}
\big/
\langle F_{1},\dots,F_{\NN}
\rangle
:=s
<
\infty.
\]

An important problem with the termination of Kohn's algorithm is its {\sl effectiveness}. Throughout this paper, we shall say that a certain quantity is {\sl effective}  if it can be expressed in terms of $s$. In~\cite{Catlin-DAngelo-2010},  John P. D'Angelo and David W. Catlin proved the effective termination of the algorithm for triangular systems, and raised an example where the termination fails to be effective.


In \cite{Siu-2010}, Yum-Tong Siu proved the effective termination of Kohn's algorithm from the view of local intersection theory in several complex variables to create multipliers with effective multiplicities. Based on his method, both the number of steps taken to terminate the algorithm, and the regularity of the Dolbeault laplacian are effective for special domains in $\mathbb{C}^{n}$. Here, we will follow the exposition in \cite{Siu-2010} for the case of dimension $3$. 

Let us briefly outline Siu's method. The first step of Kohn's algorithm allows only a linear combination of the $F_{i}$. One idea is to create {\sl generic} linear combinations
\[
A=\sum_{i=1}^{\NN}\lambda_{i}F_{i},\qquad
B=\sum_{i=1}^{\NN}\mu_{i}F_{i}
\]
whose intersection multiplicity
\[
\dim_{\mathbb{C}}\ 
\mathcal{O}_{\mathbb{C}^{2},0}
\big/
\langle A,B\rangle
\]
has an effective upper bound. The next step in Kohn's algorithm consists of taking the Jacobian $\Jac(A,B)$, and of letting  $\tilde{h}_{2}$ be the reduction of $\Jac(A,B)$. The holomorphic function $\tilde{h}_{2}$ has an effective multiplicity. Furthermore, there exists another generic linear combination $h_{1}:=\sum_{i=1}^{\NN}c_{i}F_{i}$ such that 
\[
\dim_{\mathbb{C}}\ 
\mathcal{O}_{\mathbb{C}^{2},0}
\big/
\langle \tilde{h}_{2},\ h_{1}\rangle
\]
has an effective upper bound. From $h_{1}$ and $\tilde{h}_{2}$, one may construct a holomorphic function $h_{2}(z_{1},z_{2})$ with an effective vanishing order $\lambda$ for $z_{2}\longmapsto h_{2}(0,z_{2})$ so that $h_{2}(h_{1}(z_{1},z_{2}),z_{2})$ is a subelliptic multiplier, hence up to a multiplication by a unit, $h_{2}(h_{1},z_{2})$ may be written as a Weierstrass polynomial 
\[
h_{2}(h_{1},z_{2})
=
z_{2}^{\lambda}
+
\sum_{1\leqslant j\leqslant \lambda}
a_{j}(h_{1})z_{2}^{\lambda-j}.
\]
The rest of the argument consists of applying Kohn's algorithm to the pre-multiplier $h_{1}$ together with the subelliptic multiplier $h_{2}(h_{1},z_{2})$. The algorithm terminates with an effective number of steps, because $h_{2}(h_{1},z_{2})$ is a polynomial with an effective degree $\lambda$. 

The purpose of this paper is to review some of the useful concepts in several complex variables, and to describe in greater detail the generic conditions wherever they appear. Then we will make explicit the effective constants and upper bounds that are found during the course of creating subelliptic multipliers, and we will apply these results to explicitly describe the regularity of the Dolbeault laplacian. The following is our main result:

\begin{Theorem}
Let $(z_{1},z_{2},z_{3})$ be holomorphic coordinates in $\mathbb{C}^{3}$ with $z_{i}=x_{i}+\isqrt y_{i}$. For some $\NN\geqslant 2$, let $F_{1}(z_{1},z_{2})$,\dots,$F_{\NN}(z_{1},z_{2})$ be holomorphic function germs in $\mathcal{O}_{\mathbb{C}^{2},0}$ vanishing at the origin such that 
\[
\dim_{\mathbb{C}}\ 
\mathcal{O}_{\mathbb{C}^{2},0}
\big/
\langle F_{1},\dots,F_{\NN}\rangle
:=
s
<\infty.
\]

Let $\Omega\subset \mathbb{C}^{3}$ be the domain defined by
\[
\Omega
=
\left\{
(z_{1},z_{2},z_{3})\in\mathbb{C}^{3}:\ 
2{\rm Re}\ z_{3}+
\sum_{i=1}^{\NN}|F_{i}(z_{1},z_{2})|^{2}
<0
\right\}.
\]

Then by Siu's method, Kohn's algorithm terminates  in at most $(4s^{2}-1)s$ steps. Moreover, for any $\phi\in\mathcal{D}_{0,1}(\Omega)$ with compact support,
\[
\||\phi|\|_{\varepsilon}^{2}
\lesssim
\|\overline{\partial}\phi\|^{2}
+
\|\overline{\partial}^{*}\phi\|^{2}
+
\|\phi\|^{2},
\]
where
\[
\varepsilon
\geqslant 
\frac{1}{2^{(4s^{2}-1)s+3}s^{2}(4s^{2}-1)^{4}\binom{8s+1}{8s-1}}.
\]
(See the next section for the definitions of $\mathcal{D}_{0,1}(\Omega)$ and of the tangential Sobolev norm $\||\cdot|\|_{\varepsilon}^{2}$.)

\end{Theorem}

\textbf{Acknowledgement:} This paper was written as part of the author's Ph.D thesis, who is grateful to his advisor Professor Jo\"{e}l Merker for suggesting this topic.

\section{The $\overline{\partial}$-Neumann Problem}

Let $\Omega\subseteq\mathbb{C}^{n}$ be an open domain in $\mathbb{C}^{n}$, and let $\mathcal{E}^{p,q}(\Omega)$ denote the set of smooth $(p,q)$-forms on $\Omega$. More explicitly, every element $\phi\in\mathcal{E}^{p,q}(\Omega)$ can be written in the form
\[
\phi
=
\sum_{\substack{(i_{1},\dots,i_{p})\in\mathbb{N}^{p},\\
1\leqslant i_{1}<\cdots <i_{p}\leqslant n}}\ 
\sum_{\substack{(j_{1},\dots,j_{q})\in\mathbb{N}^{q},\\ 1\leqslant j_{1}<\cdots <j_{q}\leqslant n}}
\phi_{i_{1},\dots,i_{p},j_{1},\dots, j_{q}}\
dz_{i_{1}}\wedge\cdots\wedge 
dz_{i_{p}}\wedge
d\overline{z}_{j_{1}}
\wedge\cdots\wedge
d\overline{z}_{j_{q}},
\]
where $\phi_{i_{1},\dots,i_{p},j_{1},\dots, j_{q}}\in\mathcal{E}^{0,0}(\Omega)=C^{\infty}(\Omega)$. For notational convenience, $\phi$ may be written as
\begin{eqnarray}
\label{eqn-a-standard-(p,q)-form}
\phi=
{\sum_{|I|=p}}'
{\sum_{|J|=q}}'
\phi_{IJ}\ 
dz_{I}\wedge d\overline{z}_{J}.
\end{eqnarray}
The notation $\sum'$ denotes the sum over increasing indices.
\begin{Definition}
Let $\mathcal{E}^{p,q}(\overline{\Omega})$ denote the following subset of $\mathcal{E}^{p,q}(\Omega)$:
\begin{eqnarray*}
\mathcal{E}^{p,q}(\overline{\Omega})
:=
&\bigg\{&
\phi\in \mathcal{E}^{p,q}(\Omega):\ 
\text{there exists a neighbourhood }V \text{ of }\overline{\Omega}\\
&&\text{and a smooth }\tilde{\phi}\in\mathcal{E}^{p,q}(V)\text{ such that }\tilde{\phi}|_{V}=\phi
\bigg\}.
\end{eqnarray*}
\end{Definition}

For any $f$, $g$ in $L^{2}_{(p,q)}(\Omega)$
\[
f={\sum_{|I|=p}}'{\sum_{|J|=q}}'f_{IJ}\ dz_{I}\wedge d\overline{z}_{J},\qquad
\text{and}
\qquad 
g={\sum_{|I|=p}}'{\sum_{|J|=q}}'g_{IJ}\ dz_{I}\wedge d\overline{z}_{J}
\]
with
\[
{\sum_{|I|=p}}'{\sum_{|J|=q}}'
\int_{\Omega}|f_{IJ}|^{2}\ dV<\infty
\qquad
\text{and}
\qquad
{\sum_{|I|=p}}'{\sum_{|J|=q}}'
\int_{\Omega}|g_{IJ}|^{2}\ dV<\infty,
\]
the metric $(-,-)$ on $L^{2}_{(p,q)}(\Omega)$ is defined by 
\[
(f,g):=
{\sum_{|I|=p}}'{\sum_{|J|=q}}'
\int_{\Omega}f_{IJ}\, \overline{g_{IJ}}\ d\lambda.
\]
Here $d\lambda$ denotes the Lebesgue measure on $\mathbb{C}^{n}$.

\subsection{The $\overline{\partial}$ Operator}
Let $\phi\in\mathcal{E}^{p,q}(\Omega)$ as in equation \eqref{eqn-a-standard-(p,q)-form}. The differential operator $\bar{\partial}$ is then a map $\bar{\partial}:\mathcal{E}^{p,q}(\Omega)\rightarrow \mathcal{E}^{p,q+1}(\Omega)$ defined by
\begin{eqnarray*}
\bar{\partial}\phi
&=&
\bar{\partial}
\left(
{\sum_{|I|=p}}'{\sum_{|J|=q}}'
\phi_{IJ}\
dz_{I}\wedge d\overline{z}_{J}
\right)\\
&=&
{\sum_{|I|=p}}'{\sum_{|J|=q}}'
\sum_{j=1}^{n}
\frac{\partial\phi_{IJ}}{\partial \overline{z}_{j}}\
d\overline{z}_{j}\wedge
dz_{I}\wedge
d\overline{z}_{J}.
\end{eqnarray*}

\begin{Definition}
Let $X$ and $Y$ be Banach spaces. An {\sl unbounded operator} $T$ on $X$ with target in $Y$ is consists of a linear subspace $\Dom(T)$ called the {\sl domain of $T$}, and a linear map
\[
T:\Dom(T)\rightarrow Y.
\]
The unbounded operator $T$ will be written as 
\[
(T,\Dom(T)):X\rightarrow Y.
\]
\end{Definition}

\begin{Definition}
Let $X$ and $Y$ be Banach Spaces. An unbounded operator $(T,\Dom(T)):X\rightarrow Y$ is {\sl closed} if the graph of $T$ is closed.
\end{Definition}

\begin{Definition}
Let $X$ and $Y$ be Banach spaces, and let 
\[
(T,\Dom(T)):X\rightarrow Y
\]
be an unbounded operator. Then the unbounded operator $T$ is {\sl densely defined} if $\Dom(T)$ is dense in $X$.
\end{Definition}

Note that even if $\phi\in L_{p,q}^{2}(\Omega)$, one may still define the $(p,q+1)$ form $\overline{\partial}\phi$ in the sense of currents. The space of $(p,q+1)$-currents contains the space $L_{p,q+1}^{2}(\Omega)$.

\begin{Definition}
Let $\overline{\partial}$ be the operator as above. Then $\Dom_{p,q}(\overline{\partial})$ denotes the following linear subspace of $L_{p,q}^{2}(\Omega)$:
\[
\Dom_{p,q}(\overline{\partial})
:=
\{\phi\in L_{p,q}^{2}(\Omega):\ 
\overline{\partial}\phi\in L_{p,q+1}^{2}(\Omega)\}.
\]
\end{Definition}
Clearly since $\Dom_{p,q}(\overline{\partial})$ contains the space of all $(p,q)$ forms on $\Omega$ with compact support, which forms a dense set in $L_{p,q}^{2}(\Omega)$, hence $\Dom_{p,q}(\overline{\partial})$ is a dense set.

The pair 
\[
(\overline{\partial},\Dom_{p,q}(\overline{\partial})):L_{p,q}^{2}(\Omega)
\rightarrow
L_{p,q+1}^{2}(\Omega)
\]
defines a closed, densely defined unbounded operator.

\subsection{The Hilbert Space adjoint of $\overline{\partial}$}
Before we define the Hilbert space adjoint $\overline{\partial}^{*}$ of $\overline{\partial}$, we first specify the domain of $\overline{\partial}^{*}$.

\begin{Definition}[$\Dom_{p,q}(\overline{\partial}^{*})$]
Let $\Dom_{p,q}(\overline{\partial}^{*})$ be the following linear subspace of $L_{p,q}^{2}(\Omega)$ :
\begin{eqnarray*}
&& \Dom_{p,q}(\overline{\partial}^{*})\\
&:=&
\left\{
\phi\in L_{p,q}^{2}(\Omega):\text{ the map }T_{\phi}:\Dom_{p,q}(\overline{\partial}^{*})\rightarrow \mathbb{C}\text{ defined by } T_{\phi}(u)=(\phi,\overline{\partial}u)\text{ is continuous}
\right\}.
\end{eqnarray*}
\end{Definition}

From the definition of the domain $\Dom_{p,q}(\overline{\partial}^{*})$, the action of $\overline{\partial}^{*}$ on $\Dom_{p,q}(\overline{\partial}^{*})$ may be defined as follows: let $T_{\phi}$ be the map in the definition. If $\phi\in\Dom_{p,q}(\overline{\partial}^{*})$, then linear map
\begin{eqnarray*}
T_{\phi}: \Dom_{p,q}(\overline{\partial}^{*}) &\rightarrow & 
\mathbb{C}\\
u &\mapsto & (\phi,\overline{\partial} u)
\end{eqnarray*}
is continuous on the subspace $\Dom_{p,q}(\Omega)\subseteq L_{p,q}^{2}(\Omega)$. By the Hahn-Banach theorem, there exists an extension 
\[
\tilde{T}_{\phi}:L_{p,q}^{2}(\Omega)
\rightarrow
L_{p,q}^{2}(\Omega)
\]
of $T_{\phi}$ to the whole of Hilbert space $L_{p,q}^{2}(\Omega)$. This extension is unique since $\Dom_{p,q}(\overline{\partial}^{*})$ is dense. Also, $\tilde{T}_{\phi}$ is a continuous linear operator. By Riesz representation theorem, there exists the unique element $\overline{\partial}^{*}\phi$ such that for all $u\in L_{p,q}^{2}(\Omega)$,
\[
\tilde{T}_{\phi}(u)
=
(\overline{\partial}^{*}\phi,u).
\]
If $u\in \Dom_{p,q}(\overline{\partial})$, then
\[
(\overline{\partial}^{*}\phi,u)
=
\tilde{T}_{\phi}(u)
=
T_{\phi}(u)
=
(\phi,\overline{\partial}u).
\] 

\begin{Definition}[Hilbert space adjoint of $\overline{\partial}$]
The Hilbert space adjoint $\overline{\partial}^{*}$ of $\overline{\partial}$ is an unbounded operator
\[
(\overline{\partial}^{*},\Dom_{p,q}\overline{\partial}^{*}):
L_{p,q}^{2}(\Omega)
\rightarrow
L_{p,q-1}^{2}(\Omega)
\]
such that for all $\phi\in\Dom_{p,q}(\overline{\partial}^{*})$ and $u\in \Dom_{p,q-1}(\overline{\partial})$, 
\[
(\overline{\partial}^{*}\phi,u)
=
(\phi,\overline{\partial}u).
\]
\end{Definition}
\subsection{The Dolbeault Laplacian $\Delta_{\overline{\partial}}$.}
The following unbounded operators $\overline{\partial}$ and $\overline{\partial}^{*}$ act in the following way
\[
\xymatrix{
L_{p,q-1}^{2}(\Omega) \ar@<1ex>[r]^-{\overline{\partial}} & 
L_{p,q}^{2}(\Omega) \ar@<1ex>[l]^-{\overline{\partial}^{*}}\ar@<1ex>[r]^-{\overline{\partial}}& 
L_{p,q+1}^{2}(\Omega)\ar@<1ex>[l]^-{\overline{\partial}^{*}}.
}
\]
As a result, there is an unbounded operator 
\[(\Delta_{\overline{\partial}},
\Dom_{p,q}(\Delta_{\overline{\partial}})):
L_{p,q}^{2}(\Omega)
\rightarrow
L_{p,q}^{2}(\Omega)
\]
called the Dolbeault laplacian
\[
\Delta_{\overline{\partial}}
=
\overline{\partial}\overline{\partial}^{*}
+
\overline{\partial}^{*}\overline{\partial},
\]
defined on
\begin{eqnarray*}
&& \Dom_{p,q}(\Delta_{\overline{\partial}})\\
&=&
\left\{
\phi\in L_{p,q}^{2}(\Omega):\ 
\phi\in\Dom_{p,q}(\overline{\partial})\cap
\Dom_{p,q}(\overline{\partial}^{*}),\ 
\overline{\partial}\phi\in \Dom_{p,q+1}(\overline{\partial}^{*}),\ 
\overline{\partial}^{*}\phi\in 
\Dom_{p,q-1}(\overline{\partial})
\right\}.
\end{eqnarray*}

\subsection{The Subelliptic estimate and Subelliptic multipliers}

\begin{Definition}[$\mathcal{D}_{p,q}(\Omega)$] The set $\mathcal{D}_{p,q}(\Omega)$ is defined to be
\[
\mathcal{D}_{p,q}(\Omega)
=
\Dom_{p,q}(\overline{\partial}^{*})
\cap
\mathcal{E}^{p,q}(\overline{\Omega}).
\]
\end{Definition}

\begin{Definition}[The Tangential Sobolev Norm]
Let $f(t_{1},\dots,t_{2n-1},r)\in\mathcal{S}(\mathbb{R}^{2n-1}\times\mathbb{R})$. The pseudodifferential operator of order $s$, denoted by $\Lambda^{s}$, is defined by
\[
\Lambda^{s}f
=
\frac{1}{(2\pi)^{\frac{2n-1}{2}}}
\int_{r=-\infty}^{0}
\int_{\mathbb{R}^{2n-1}}
e^{\isqrt\sum_{k=1}^{2n-1}t_{k}\tau_{k}}
\left(1+\sum_{k=1}^{2n-1}|\tau_{k}|^{2}\right)^{s/2}
\hat{f}(\tau_{1},\dots,\tau_{2n-1},r)d\tau\,dr,
\]
where
\[
\hat{f}(\tau_{1},\dots,\tau_{2n-1},r)
=
\frac{1}{(2\pi)^{\frac{2n-1}{2}}}
\int_{\mathbb{R}^{2n-1}}
e^{-\isqrt\sum_{k=1}^{2n-1}t_{k}\tau_{k}}
f(t_{1},\dots,t_{2n-1},r)dt.
\]
The tangential sobolev norm $\||\bullet |\|_{s}^{2}$ is defined by
\[
\||f|\|_{s}^{2}
=
\frac{1}{(2\pi)^{\frac{2n-1}{2}}}
\int_{r=-\infty}^{0}
\int_{\mathbb{R}^{2n-1}}
\left(1+\sum_{k=1}^{2n-1}|\tau_{k}|^{2}\right)^{\frac{s}{2}}
|\hat{f}(\tau_{1},\dots,\tau_{2n-1},r)|^{2}
d\tau\,dr.
\]
\end{Definition}

By \cite[Appendix,~Proposition~A.3.1]{Folland-Kohn-1972}, if $s> s'$, then for any $\varepsilon>0$, there exists a neighbourhood $V\subset \mathbb{R}^{2n}$ of the origin such that $\||u|\|_{s'}\leqslant \varepsilon \||u|\|_{s}$ for all $u$ supported in $V$.

\begin{Definition}[The Subelliptic Estimates]
Suppose that $\Omega\subset\subset\mathbb{C}^{n}$ is an open domain whose closure is compact, and whose boundary is smooth. Let $x\in\overline{\Omega}$. The $\overline{\partial}$-Neumann problem satisfies a subelliptic estimate on $(0,q)$ forms if there exists a neighbourhood $U\subseteq\mathbb{C}^{n}$ of $x$, and positive constants $c$, $\varepsilon$, such that for all $\phi\in\mathcal{D}_{0,q}(U\cap \Omega)$ with compact support,
\[
\||\phi|\|_{\varepsilon}^{2}
\leqslant
c\left(
\|\overline{\partial}\phi\|^{2}
+
\|\overline{\partial}^{*}\phi\|^{2}
+
\|\phi\|^{2}
\right).
\]
\end{Definition}

From here, we will adopt the following notation: we let $Q(\phi,\psi)$ denote the quadratic form
\[
Q(\phi,\psi)
=
(\overline{\partial}\phi,
\overline{\partial}\psi)
+
(\overline{\partial}^{*}\phi,
\overline{\partial}^{*}\psi)
+
(\phi,\psi).
\]

\begin{Definition}[Subelliptic multipliers]
Let $\Omega$ be a smoothly bounded pseudoconvex domain in $\mathbb{C}^{n}$. Let $x\in\overline{\Omega}$ be a point, and let $\mathcal{C}_{x}^{\infty}$ be the ring of germs of smooth functions at $x$. An element $g\in\mathcal{C}_{x}^{\infty}$ is a subelliptic multiplier on $(0,1)$ forms if there exists a neighbourhood $U\subseteq\mathbb{C}^{n}$ of $x$, and positive constants $c$, $\varepsilon$, such that for all $\phi\in \mathcal{D}_{0,q}(U\cap \Omega)$ with compact support,
\[
\||g\phi|\|_{\varepsilon}^{2}
\leqslant
cQ(\phi,\phi).
\]
\end{Definition}

\section{Kohn's Algorithm for subelliptic multipliers}

Let $(z_{1},\dots,z_{n},x_{n+1}+\isqrt y_{n+1})$ be a holomorphic coordinates of $\mathbb{C}^{n+1}$. Let $F_{1},\dots,F_{\NN}$ be holomorphic function germs vanising at the origin in $\mathbb{C}^{n+1}$. For convenience, we let $z:=(z_{1},\dots,z_{n})$. Let $r(z,z_{n+1},\overline{z},\overline{z}_{n+1})$ be the real analytic function defined by
\[
r(z,z_{n+1},\overline{z},\overline{z}_{n+1})
=
2\text{Re}(z_{n+1})+\sum_{j=1}^{\NN}|F_{j}(z)|^{2}
=
x_{n+1}+\sum_{j=1}^{\NN}|F_{j}(z)|^{2}.
\]
Let $\Omega$ be the open domain defined by
\[
\Omega=\{r<0\},
\]
and the boundary $b\Omega$ is the following set
\[
b\Omega=
\{r=0\}
\]
which is smooth.
Clearly, $0\in b\Omega$.

\begin{Definition}[Kohn's Algorithm for Special Domains]
Let $\mathcal{I}_{0}:=\langle F_{1},\dots,F_{\NN}\rangle$ be the ideal in $\mathcal{O}_{\mathbb{C}^{n},0}$ generated by the holomorphic function germs $F_{i}$'s. We associate with $\mathcal{I}_{0}$ with a sequence of radical ideals 
\[
\mathcal{I}_{1}\subseteq
\mathcal{I}_{2}\subseteq\cdots
\]
in $\mathcal{O}_{\mathbb{C}^{n},0}$ as follows:

\smallskip\noindent{\bf (i)}
Let $g_{1},\dots, g_{n}$ be linear combinations of the $F_{i}$
\[
g_{i}=\sum_{k=1}^{\NN}c_{ik}F_{k}
\eqno
{\scriptstyle{(c_{ik}\,\in\,\mathbb{C})}}.
\]
 Let $\Jac(g_{1},\dots,g_{n})$ be the Jacobian of the $g_{i}$ and define
\[
\mathcal{I}_{1}^{\#}
=
\langle
\Jac(g_{1},\dots,g_{n}):\ 
g_{i}\text{ is a linear combination of }F_{i}
\rangle.
\]
Then set $\mathcal{I}_{1}=\sqrt{\mathcal{I}_{1}^{\#}}$.
\end{Definition}

\smallskip\noindent{\bf (ii)} (Inductive Step)
Suppose that $\mathcal{I}_{k}$ has been constructed, let $\mathcal{I}_{k+1,\Jacsmall}^{\#}$ be the ideal generated by $\Jac(h_{1},\dots,h_{n})$ where each $h_{i}$ is either an element of $\mathcal{I}_{k}$, or is a linear combination $g_{k}$ of the $F_{i}$. Then let
\[
\mathcal{I}_{k+1}^{\#}
=
\mathcal{I}_{k+1,\Jacsmall}^{\#}+
\mathcal{I}_{k},
\]
and set $\mathcal{I}_{k+1}=\sqrt{\mathcal{I}_{k+1}^{\#}}$.

Here are some effects of Kohn's algorithm on $\varepsilon$ the subelliptic regularity of the multipliers.

\begin{Proposition}\label{Kohn-Property}
Let $f\in\mathcal{O}_{\mathbb{C}^{n},0}$ be a subelliptic multiplier and let $g\in\mathcal{O}_{\mathbb{C}^{n},0}$ be a holomorphic function germ. This means that at $0\in b\Omega$, there exists an open neighbourhood $U\subseteq\mathbb{C}^{n}$ of $0$, and strictly positive constants $c$, $\varepsilon$ such that for all $\phi\in\mathcal{D}_{0,1}(U\cap\Omega)$, one has
\[
\||f\phi|\|_{\varepsilon}^{2}
\leqslant 
cQ(\phi,\phi).
\]
\smallskip\noindent{\bf (i)}
Suppose there exists $N>0$ such that $|g|^{N}\leqslant |f|$, then $g$ is also a subelliptic multiplier and 
\[
\||g\phi|\|_{\varepsilon/N}^{2}
\leqslant
cQ(\phi,\phi).
\]
\smallskip\noindent{\bf (ii)}
If the $(0,1)$-form $\phi$ is written as $\phi=\sum_{k=1}^{n}\phi_{k}\ d\overline{z}_{i}$, one has
\[
\left\|\left|
\sum_{k=1}^{n}\frac{\partial f}{\partial z_{k}}\phi_{k}
\right|\right\|_{\varepsilon/2}^{2}
\leqslant c
Q(\phi,\phi).
\]
\smallskip\noindent{\bf (iii)}
Let $f_{1},\dots,f_{n}$ be subelliptic multipliers. Suppose there exists $\varepsilon>0$ such that for all $i$,
\[
\left\|\left|
\sum_{k=1}^{n}\frac{\partial f_{i}}{\partial z_{k}}\phi_{k}
\right|\right\|_{\varepsilon}^{2}
\leqslant c
Q(\phi,\phi),
\]
then 
\[
\||\Jac(f_{1},\dots,f_{n})\phi|\|_{\varepsilon}^{2}
\leqslant 
cQ(\phi,\phi).
\]
\smallskip\noindent{\bf (iv)}
For any $g\in\mathcal{O}_{\mathbb{C}^{n},0}$, one has
\[
\||gf\phi|\|_{\varepsilon}^{2}
\lesssim
Q(\phi,\phi).
\]
\end{Proposition}
\begin{proof}
For properties {\bf (i)} to {\bf (iii)}, see \cite{DAngelo-Book-1993}. For the last property, we may refer to \cite[p.~94, Proposition 4.7(D)]{Kohn-1979}\footnote{For proof, see page 97}. We will give a summary of the proof of the last property. Given $f\in\mathcal{O}_{\mathbb{C}^{n}}(U)$ for some open neighbourhood $U\subseteq \mathbb{C}^{n}$ of the origin, there exists $\varepsilon>0$ such that for all $\phi\in\mathcal{D}_{0,1}(U\cap \Omega)$, one has
\[
\||f\phi|\|_{\varepsilon}^{2}
\lesssim
Q(\phi,\phi).
\]
By remark in \cite[p~93, Section~4, Paragraph~2]{Kohn-1979}, for any $V\subseteq U$ a open subset of $U$, the same $\varepsilon>0$ will satisfy the property that for all $\phi\in\mathcal{D}_{0,1}(V\cap \Omega)$,
\[
\||f|_{V\cap \overline{\Omega}}\ \phi|\|_{\varepsilon}^{2}
\lesssim
Q(\phi,\phi).
\]
Given $g\in\mathcal{O}_{\mathbb{C}^{n},0}$, for some $V\subseteq U$, $gf\in\mathcal{O}_{\mathbb{C}^{n}}(V)$. Upon restriction to a smaller open set, $g$ is bounded on $\overline{V}$. For any $\phi\in\mathcal{D}_{0,1}(V\cap\Omega)$, its support is contained in $V\cap \overline{\Omega}$. Hence 
\[
\||gf\phi|\|_{\varepsilon}^{2}
\lesssim
\||f\phi|\|_{\varepsilon}^{2}
\lesssim 
Q(\phi,\phi).\qedhere
\]
\end{proof}
\section{Local Geometry of Complex Spaces and Local Intersection Theory}

\subsubsection{}\label{LAG-para-1} Throughout this section, we will study study the geometry of analytic varieties near the origin.

\subsubsection{ }\label{LAG-para-2} Let $\mathcal{O}_{\mathbb{C}^{n},0}$ denote the ring of holomorphic function germs at the origin. It can be canonically identified with $\mathbb{C}\{z_{1},\dots, z_{n}\}$ the ring of convergent power series.

\subsubsection{}\label{LAG-para-3} The ring $\mathcal{O}_{\mathbb{C}^{n},0}$ is local and let $\mathfrak{m}$ denote its unique maximal ideal, which can be characterised by one of the following equivalent properties:

\smallskip\noindent{\bf (i)}
\[
\mathfrak{m}:=
\{f\in\mathcal{O}_{\mathbb{C}^{n},0}:\ f(0)=0\};
\]

\smallskip\noindent{\bf (ii)}
\[
\mathfrak{m}=
\langle z_{1},\dots, z_{n}\rangle;
\]

\smallskip\noindent{\bf (iii)}
every holomorphic function germ $f$ may be written as $f=\sum_{k\geq 1}f_{k}$ a sum of homogeneous polynomials $f_{k}$ of degree $k$.

\subsubsection{}\label{LAG-para-4} For each $l\in\mathbb{N}$, we define $\mathfrak{m}^{l+1}$ recursively by
\begin{eqnarray*}
\mathfrak{m}^{l+1} &=& 
\mathfrak{m}\cdot \mathfrak{m}^{l}\\
&=&
\left\{
\sum_{k=1}^{N}f_{k}g_{k}:\ 
N\in\mathbb{N},\ 
f_{k}\in\mathfrak{m},\ 
g_{k}\in\mathfrak{m}^{l}
\right\}.
\end{eqnarray*}
For any fixed $l\geqslant 1$, the following conditions are equivalent:

\smallskip\noindent{\bf (i)} $h\in\mathfrak{m}^{l}$;

\smallskip\noindent{\bf (ii)} for each $(\alpha_{1},\dots,\alpha_{n})\in\mathbb{N}^{n}$ such that $\alpha_{1}+\cdots+\alpha_{n}\leqslant 
l-1$,
\[
(\partial_{z_{1}}^{\alpha_{1}}
\cdots
\partial_{z_{n}}^{\alpha_{n}}h)
(0)=0;
\]

\smallskip\noindent{\bf (iii)} 
\[
\mathfrak{m}^{l}=
\langle
z_{1}^{k_{1}}\cdots 
z_{n}^{k_{n}}:\ 
k_{1}+\cdots k_{n}=l
\rangle;
\]

\smallskip\noindent{\bf (iv)}  
$h=\sum_{k\geqslant l}h_{k}$ where either $h_{k}$ vanishes or is a homogeneous polynomial of degree $k$.

\subsubsection{Multiplicity} \label{LAG-para-5}
\begin{Definition}\label{LAG-def-Multiplicity}
Let $h\in\mathcal{O}_{\mathbb{C}^{n},0}$ be a holomorphic function germ, which can be written as
\[
h=\sum_{k=0}^{\infty}h_{k}
\]
a sum of homogeneous polynomials $h_{k}$ of degree $k$. The multiplicity of $h$, which will be denoted by $\mult_{0}\ h$, is the smallest positive integer $k$ for which $h_{k}\not\equiv 0$.
\end{Definition}

\subsubsection{}\label{LAG-para-6} By Paragraph \ref{LAG-para-4}((i)$\iff$(iv)), the holomorphic function $h$ lies in  $\mathfrak{m}^{l}$ if and only if $\mult_{0}\ h\geqslant l$. In section 5, we will study the  geometric characterisation of multiplicity of a holomorphic function, and the extension of this notion to certain ideals.

\subsection{Local Analytic Geometry}

\subsubsection{}\label{LAG-para-1-1} In this subsection, we let $F_{1}$,\dots, $F_{\NN}$ be holomorphic function germs in $\mathcal{O}_{\mathbb{C}^{n},0}$ vanishing at the origin.  For easier exposition, we will not specify the domain of definition of the holomorphic function germs.

\subsubsection{Local Analytic Set, Germs of Analytic Space} \label{LAG-para-1-3}

\begin{Definition}\label{LAG-def-locally analytic}
A set $X\subseteq \mathbb{C}^{n}$ is {\sl locally analytic} if for any point $p\in X$, there exists an open subset $V$ of $p$ in $\mathbb{C}$, and finitely many holomorphic functions $f_{1}$,\dots,$f_{s}$ defined on $V$, such that 
\[
X\cap V=\{x\in V:\ f_{1}(x)=\cdots=
f_{s}(x)=0\}.
\]
\end{Definition}

\begin{Definition}\label{LAG-def-Germ of Analytic Space}
A {\sl germ of analytic space} $(X,0)$ is a germ at $0$ of a locally analytic subset of $\mathbb{C}^{n}$.
\end{Definition}

\subsubsection{}\label{LAG-para-1-4} Any germ of an analytic space $(X,0)$ may be uniquely written as 
\[
(X,0)=(X_{1},0)\cup\cdots\cup (X_{r},0)
\]
a union of irreducible components\footnote{A germ of an analytic space $(X,x)$ is irreducible if whenever $(X,x)=(X_{1},x)\cup (X_{2},x)$ with $(X_{1},x)$ and $(X_{2},x)$ germs of analytic spaces, either $(X,x)=(X_{1},x)$ or $(X,x)=(X_{2},x)$.}, each of which is also a germ of an analytic space (\cite[Corollary~3.4.18, p~118]{deJong-Pfister-2000}).

\subsubsection{$(V(F),0)$, $(V(\mathcal{I}_{F}),0)$ and $\mathcal{I}(X,0)$.}\label{LAG-para-1-5}

\begin{Definition}\label{LAG-def-V(F)}
Let $F\in\mathcal{O}_{\mathbb{C}^{n},0}$. The germ of an analytic hypersurface $(V(F),0)$ is defined as follows. Let $U$ be an open neighbourhood of the origin on which $F$ seen as a power series converges. Consider $V(F)=\{p\in U:\ F(p)=0\}$. Then $(V(F),0)$ is the germ of $V(F)$ at  zero, and is called the {\sl zero set} of $F$.
\end{Definition}
 
\begin{Definition}\label{LAG-def-V(Ideal)}
Let $\mathcal{I}_{F}=\langle F_{1},\dots, F_{\NN}\rangle$ be an ideal of $\mathcal{O}_{\mathbb{C}^{n},0}$. The germ of analytic space $(V(\mathcal{I}_{F}),0)$ is defined by
\[
(V(\mathcal{I}_{F}),0)=
\bigcap_{i=1}^{\NN}(V(F_{i}),0).
\]
\end{Definition}

\begin{Definition}\label{LAG-def-I(X,0)}
Let $(X,0)$ be a germ of an analytic space. Then define
\[
\mathcal{I}(X,0)
=
\{f\in\mathcal{O}_{\mathbb{C}^{n},0}:\ (X,0)\subseteq (V(f),0)\}.
\]
\end{Definition}

\subsubsection{Properties of $(V(\mathcal{I}_{F}),0)$ and $\mathcal{I}(X,0)$}\label{LAG-para-1-6}

Let $F_{1},\dots, F_{\NN}$ and $G_{1},\dots, G_{\MM}$ be holomorphic function germs in $\mathcal{O}_{\mathbb{C}^{n},0}$. Let $\mathcal{I}_{F}=\langle F_{1},\dots, F_{\NN}\rangle$ and $\mathcal{I}_{G}=\langle G_{1},\dots,G_{\MM}\rangle$ be the corresponding ideals they generate. Let $(X,0)$ and $(Y,0)$ be germs of analytic spaces.

\smallskip\noindent{\bf (i)} 
 $\mathcal{I}_{F}\subseteq \mathcal{I}_{G}$ implies that $(V(\mathcal{I}_{G}),0)\subseteq
(V(\mathcal{I}_{F}),0)$;

\smallskip\noindent{\bf (ii)} $(X,0)\subseteq (Y,0)$ implies that $\mathcal{I}(Y,0)\subseteq \mathcal{I}(X,0)$;

\smallskip\noindent{\bf (iii)} 
for any $k\in\mathbb{N}_{\geqslant 1}$, and for any ideal $\mathcal{I}_{F}$, $(V(\mathcal{I}_{F}^{k}),0)=
(V(\mathcal{I}_{F}),0)$;

\smallskip\noindent{\bf (iv)} for any germ of analytic space $(X,0)$, $(V(\mathcal{I}(X,0)),0)=(X,0)$

\smallskip\noindent{\bf (v)} (Nullstellensatz) $\mathcal{I}(V(\mathcal{I}_{F}),0))=
\sqrt{\mathcal{I}_{F}}$.

For ease of notation, let $V(F_{1},\dots, F_{\NN})$ or $V(\mathcal{I}_{F})$ denote $(V(\mathcal{I}_{F}),0)$.

\subsection{Local Intersection Theory I}

\subsubsection{ }\label{LAG-para-2-1} We begin with the characterisation of complete intersections of germs of analytic varieties at the origin.

\begin{Theorem}\label{LAG-thm-characterisations of complete intersection 1}
Let $F_{1}$,\dots, $F_{\NN}$ be holomorphic function germs in $\mathcal{O}_{\mathbb{C}^{n},0}$ at the origin. The following statements are equivalent.

\smallskip\noindent{\bf (i)} 
$V(F_{1},\dots,F_{\NN})=\{0\}$;

\smallskip\noindent{\bf (ii)} there exists a positive integer $q\geqslant 1$ such that $\mathfrak{m}^{q}\subseteq \mathcal{I}_{F}$;

\smallskip\noindent{\bf (iii)} 
the number
\[
\dim_{\mathbb{C}}\ 
\mathcal{O}_{\mathbb{C}^{n},0}
\big/
\mathcal{I}_{F}
=: s
\]
is finite;

\smallskip\noindent{\bf (iv)} there exists a positive integer $p$ such that locally  
\[
|z|^{p}\lesssim
\sum_{i=1}^{\NN}|F_{i}|.
\]
\end{Theorem}
\begin{proof}
The proof proceeds in the following manner: (i)$\implies$(ii)$\implies$ (iii)$\implies$(i), and (ii)$\iff$(iv).

For (i)$\implies$(ii), since $V(F_{1},\dots,F_{\NN})=\{0\}=V(\mathfrak{m})$, there is an equality of ideals $\mathcal{I}(V(F_{1},\dots,F_{\NN}))
=
\mathcal{I}(V(\mathcal{\mathfrak{m}}))$. By Nullstellensatz, therefore  $\mathfrak{m}=\sqrt{\mathfrak{m}}= \sqrt{\mathcal{I}_{F}}$. Hence there exists $q\in\mathbb{N}_{\geqslant 1}$ such that $\mathfrak{m}^{q}\subseteq\mathcal{I}_{F}$.

For (ii)$\implies$ (iii), the condition that $\mathfrak{m}^{q}\subseteq\mathcal{I}_{F}$ implies that there is a surjective map of $\mathbb{C}$-vector space
\begin{eqnarray*}
\mathcal{O}_{\mathbb{C}^{n},0}
\big/\mathfrak{m}^{q}
&\longrightarrow &
\mathcal{O}_{\mathbb{C}^{n},0}
\big/\mathcal{I}_{F}\\
f\ \mod\ \mathfrak{m}^{q} &\longmapsto & 
f\ \mod\ \mathcal{I}_{F}.
\end{eqnarray*}
Hence, 
\[
\dim_{\mathbb{C}}\ \mathcal{O}_{\mathbb{C}^{n},0}
\big/
\mathcal{I}_{F}
\leqslant
\dim_{\mathbb{C}}\ \mathcal{O}_{\mathbb{C}^{n},0}
\big/
\mathfrak{m}^{q},
\]
and the proof is complete since $\dim_{\mathbb{C}}
\mathcal{O}_{\mathbb{C}^{n},0}\big/
\mathfrak{m}^{q}$ is always finite for $q\in\mathbb{N}_{\geqslant 1}$.

For (iii)$\implies $ (i), it is needed to show that the set 
\[
\{(\alpha_{1},\dots,\alpha_{n})\in\mathbb{C}^{n}:\ 
F_{k}(\alpha_{1},\dots,\alpha_{n})=0
\text{ for all }1\leqslant k\leqslant \NN\}
\]
is finite. To this effect, it suffices to show that there can only be finitely many choices for each $\alpha_{i}$. Since $
\mathcal{O}_{\mathbb{C}^{n},0}
\big/\mathcal{I}_{F}$ is finite dimensional, for each $1\leqslant i\leqslant n$, there exists $k_{i}\in\mathbb{N}_{\geqslant 1}$ such that the classes
\[
\{
1,\ z_{i},\ \dots,\ z_{i}^{k_{i}}
\}
\]
form a linearly dependent set in $\mathcal{O}_{\mathbb{C}^{n},0}
\big/
\mathcal{I}_{F}$. Hence there exist constants $\{c_{i,0},\dots,c_{i,k_{i}}\}$ such that 
\[
\sum_{j=0}^{k_{i}}c_{ij}z_{i}^{j}
\equiv 
0\ \mod\ \mathcal{I}_{F}.
\]
Thus there exists a holomorphic function $h_{i}(z_{1},\dots,z_{n})\in\mathcal{I}_{F}$ such that 
\[
\sum_{j=0}^{k_{i}}
c_{ij}z_{i}^{j}
-
h_{i}(z_{1},\dots,z_{n})
\equiv 0
\eqno
{\scriptstyle{(1\,\leqslant\, i\,\leqslant\,n)}}.
\]
If $(\alpha_{1},\dots,\alpha_{n})\in\mathcal{I}_{F}$, then for all $1\leqslant i\leqslant n$, one has $h_{i}(\alpha_{1},\dots,\alpha_{n})=0$. Hence
\[
\sum_{j=0}^{k_{i}}c_{ij}\alpha_{i}^{j}
=
0.
\]
The equation above is a polynomial equation in degree $k_{i}$, and so there are at most $k_{i}$ distinct solutions for $\alpha_{i}$. This holds for all $i$, and therefore $V(\mathcal{I}_{F})$ is a finite set. The proof is complete.

The implication (ii)$\implies$ (iv) is immediate. The converse will be proved after Skoda's theorem is introduced. The proof is reproduced from \cite[p~1179]{Siu-2010}
\end{proof}

\begin{Theorem}[Theorem of Henri Skoda]\label{LAG-thm-Theorem of Henri Koda}
Let $D$ be a pseudoconvex domain in $\mathbb{C}^{n}$ and let $\chi$ be a plurisubharmonic function on $D$. Let $g_{1}$,\dots,$g_{m}$ be holomorphic functions on $D$. Let $\alpha>1$ and $l=\min\{n,m-1\}$. Then for every holomorphic function $F$ on $D$ such that 
\[
\int_{D}|F|^{2}|g|^{-2\alpha l-2}
e^{-\chi}<\infty,
\]
there exist holomorphic functions $h_{1}$,\dots,$h_{m}$ on $D$ such that 
\[
F=
\sum_{i=1}^{m}h_{i}g_{i},
\]
and
\[
\int_{D}|h|^{2}|g|^{-2\alpha l-2}
e^{-\chi}
\leqslant
\frac{\alpha}{\alpha-1}
\int_{D}|F|^{2}|g|^{-2\alpha l-2}
e^{-\chi},
\]
where 
$
|g|=
\left(
\sum_{i=1}^{m}|g_{i}|^{2}\right)^{1/2}$
 and 
$|h|=
\left(
\sum_{i=1}^{m}|h_{i}|^{2}\right)^{1/2}.
$
\end{Theorem}

\begin{proof}[Finishing the proof of Theorem] For any non-negative numbers $\gamma_{1}$,\dots,$\gamma_{n}$ with $\gamma_{1}+\cdots+\gamma_{n}=(n+2)p$, Skoda's theorem is applied with the following variables: $F=z_{1}^{\gamma_{1}}\cdots z_{n}^{\gamma_{n}}$, $m=\NN+n$, $\chi\equiv 0$, $(F_{1},\dots, F_{\NN},0,\dots, 0)=(g_{1},\dots,g_{m})$, $l=n$ and $\alpha=\frac{n+1}{n}$. By the hypothesis in (iv), 
\[
|z_{1}^{\gamma_{1}}\cdots z_{n}^{\gamma_{n}}|^{2}
\lesssim
|z|^{2(n+2)p}
\lesssim
\left(
\sum_{i=1}^{\NN}|F_{i}|
\right)^{2(n+2)}
\lesssim
\left(
\sum_{i=1}^{\NN}|F_{i}|^{2}
\right)^{(n+2)}
\]
where the last inequality follows from Jensen's inequality. Hence over a small pseudoconvex domain $D$,
\begin{eqnarray*}
\int_{D}
\frac{|z_{1}^{\gamma_{1}}\cdots z_{n}^{\gamma_{n}}|^{2}}{\left(\sum_{j=1}^{2}|F_{j}|^{2}\right)^{n+2}}
\lesssim  
\int_{D}
\frac{\left(\sum_{i=1}^{\NN}|F_{i}|^{2}
\right)^{(n+2)}}{\left(\sum_{i=1}^{\NN}|F_{i}|^{2}
\right)^{(n+2)}}
=
\int_{D}1 <\infty.
\end{eqnarray*}
Skoda's theorem applies and therefore $z_{1}^{\gamma_{1}}\cdots z_{n}^{\gamma_{n}}\in\mathcal{I}_{F}$ . Consequently, $\mathfrak{m}^{(n+2)p}\subseteq \mathcal{I}_{F}$.
\end{proof}

From the proof above, we obtain the following corollary.

\begin{Corollary}\label{LAG-cor-of Henri Skoda}
Let $F_{1}$,\dots, $F_{\NN}$ be holomorphic function germs in $\mathcal{O}_{\mathbb{C}^{n},0}$ at the origin, and suppose there exists $p\in\mathbb{N}_{\geqslant 1}$ such that 
\[
|z|^{p}\lesssim
\sum_{i=1}^{\NN}|F_{i}|
\]
in a small neighbourhood $0$, then $\mathfrak{m}^{(n+2)p}\subseteq \mathcal{I}_{F}$.
\end{Corollary}

\subsubsection{The intersection invariants $(p,q,s)$.}\label{LAG-para-2-2}

\begin{Definition}
Let $F_{1}$,\dots, $F_{\NN}$ be holomorphic function germs in $\mathcal{O}_{\mathbb{C}^{n},0}$ at the origin. The ideal $\mathcal{I}_{F}=\langle F_{1},\dots, F_{\NN}\rangle$ is said to have finite intersection multiplicity with data $(p,q,s)$ if

\smallskip\noindent{\bf (i)}  
$p$ is the smallest strictly postive integer satsifying
\[
|z|^{p}
\lesssim
\sum_{i=1}^{\NN}|F_{i}|;
\]

\smallskip\noindent{\bf (ii)}
$q$ is the smallest strictly positive integer satisfying 
\[
\mathfrak{m}^{q}\subseteq
\mathcal{I}_{F}; 
\] 
\smallskip\noindent{\bf (iii)} 
$s$ is following number below
\[
\dim_{\mathbb{C}}\ 
\mathcal{O}_{\mathbb{C}^{n},0}
\big/
\mathcal{I}_{F}=:s.
\]
\end{Definition}

\subsubsection{The relations between the intersection invariants.} \label{LAG-para-2-3}
\begin{Proposition}\label{LAG-prop-Relations between the intersection invariants}
Let $F_{1}$,\dots, $F_{\NN}$ be holomorphic function germs in $\mathcal{O}_{\mathbb{C}^{n},0}$ at the origin so that the ideal they generate $\mathcal{I}_{F}$ has finite intersection multiplicity with data $(p,q,s)$. Then we have the following inequalities:

\smallskip\noindent{\bf (i)}
$q\leqslant s\leqslant \binom{n+q-1}{q-1}$,

\smallskip\noindent{\bf (ii)}
$p\leqslant q\leqslant (n+2)p$.  
\end{Proposition}
\begin{proof}
To prove $q\leqslant s$, it is first observed that $\mathcal{O}_{\mathbb{C}^{n},0}
\big/\mathcal{I}_{F}$ is also a local ring with the maximal ideal $\mathfrak{m}\big/\mathcal{I}_{F}$. In the chain of inclusion of vector spaces with 
\[
\frac{\mathcal{O}_{\mathbb{C}^{n},0}}{\mathcal{I}_{F}}
\supseteq 
\frac{\mathfrak{m}}{\mathcal{I}_{F}}
\supseteq
\left(
\frac{\mathfrak{m}}{\mathcal{I}_{F}}
\right)^{2}
\supseteq \cdots\
\supseteq
\left(
\frac{\mathfrak{m}}{\mathcal{I}_{F}}
\right)^{s+1},
\]
since $\mathcal{O}_{\mathbb{C}^{n},0}\big/
\mathcal{I}_{F}$ is an $s$-dimensional complex vector space,  there exists a positive integer $1\leqslant k\leqslant s$ such that 
\[
\left(
\frac{\mathfrak{m}}{\mathcal{I}_{F}}
\right)^{k}
=
\left(
\frac{\mathfrak{m}}{\mathcal{I}_{F}}
\right)^{k+1}
=
\left(
\frac{\mathfrak{m}}{\mathcal{I}_{F}}
\right)
\left(
\frac{\mathfrak{m}}{\mathcal{I}_{F}}
\right)^{k}.
\]
By Nakayama's lemma\footnote{The following version of Nakayama's lemma is used: let $A$ be a commutative local ring with $1$, and $\mathfrak{m}$ its maximal ideal. For any finitely generated $A$-module $M$, if $\mathfrak{m}M=M$, then $M=0$}, 
\[
\left(
\frac{\mathfrak{m}}{\mathcal{I}_{F}}
\right)^{k}
\equiv 0
\qquad
\text{ in }
\frac{\mathcal{O}_{\mathbb{C}^{n},0}}{\mathcal{I}_{F}}.
\]
Therefore, if $g_{1}$,\dots,$g_{k}$ are elements of $\mathfrak{m}$ in $\mathcal{O}_{\mathbb{C}^{n},0}$, then the class $g_{1}\cdots g_{k}$ belongs to $(\mathfrak{m}\big/\mathcal{I}_{F})^{k}$
which is the zero vector space. Hence the holomorphic function  $g_{1}\cdots g_{k}$ lies in $\mathcal{I}_{F}$. Since the  $g_{1}\cdots g_{k}$ generate $\mathfrak{m}^{k}$, the ideal $\mathfrak{m}^{k}$ is contained in $ \mathcal{I}_{F}$. By the definition of $q$, the inequality $q\leqslant k\leqslant s$ holds.

Next, for $s\leqslant \binom{n+q-1}{q-1}$, this follows directly from $\mathfrak{m}^{q}\subseteq\mathcal{I}_{F}$.

In the second set of inequalities, to prove $p\leqslant q$, observe that since $\mathfrak{m}^{q}\subseteq \mathcal{I}_{F}$,
\[
|z|^{q}\lesssim \sum_{i=1}^{\NN}|F_{i}|.
\]
Hence, by the definition of $p$,  $p\leqslant q$. 

To prove $q\leqslant (n+2)p$, it follows from Corollary \ref{LAG-cor-of Henri Skoda}.
\end{proof}

\subsubsection{Application of the relations of the invariants.}\label{LAG-para-2-4}

\begin{Lemma}\label{LAG-lem-if h(0)=0 then h^s is in the ideal}
Let $F_{1}$, \dots, $F_{\NN}$ be holomorphic function germs such that the intersection multiplicity of $\mathcal{I}_{F}=\langle F_{1},\dots, F_{\NN}\rangle$ is finite with data $(p,q,s)$. If $h\in\mathcal{O}_{\mathbb{C}^{n},0}$ is a holomorphic function germ with $h(0)=0$, then $h^{s}\in\mathcal{I}_{F}$.
\end{Lemma}

\begin{proof}
Since $h(0)=0$, the function $h$ lies in $\mathfrak{m}$. Consequently, $h^{s}\in\mathfrak{m}^{s}$. By $q\leqslant s$, there is an inclusion of ideals $\mathfrak{m}^{s}\subseteq \mathfrak{m}^{q}$. Therefore, $h^{s}\in\mathfrak{m}^{s}\subseteq
\mathfrak{m}^{q}\subseteq\mathcal{I}_{F}$.
\end{proof}

\subsection{Local Intersection Theory II}

\subsubsection{}\label{LAG-para-3-1} The case where $\NN=n=\dim\ \mathbb{C}^{n}$ brings another set of equivalent conditions for complete local intersection of $n$ holomorphic function germs $F_{1}$, \dots, $F_{\NN=n}$.

\subsubsection{}\label{LAG-para-3-2}
\begin{Theorem}\label{LAG-thm-characterisations of complete intersections 2}
Let $F_{1}$, \dots, $F_{n}$ be holomorphic function germs in $\mathcal{O}_{\mathbb{C}^{n},0}$ such that $F_{i}(0)=0$ for all $1\leqslant i\leqslant n$. The following are equivalent:

\smallskip\noindent{\bf (i)} 
$\dim_{\mathbb{C}}\ 
\mathcal{O}_{\mathbb{C}^{n},0}
\big/
\langle F_{1},\dots,F_{n}\rangle
=: s<\infty;$

\smallskip\noindent{\bf (ii)} the holomorphic map of germs of analytic spaces 
\begin{eqnarray*}
F:(\mathbb{C}^{n},0) &\longrightarrow & 
(\mathbb{C}^{n},0)\\
(z_{1},\dots, z_{n})
&\longmapsto & 
(F_{1},\dots,F_{n})
\end{eqnarray*}
defines a ramified $s$-sheeted analytic covering;

\smallskip\noindent{\bf (iii)}
for each $1\leqslant i\leqslant n$, let $\varepsilon_{i}$ be a small strictly positive number, and $\Gamma$ be given by
\[
\Gamma=
\{z:\ |F_{i}|=\varepsilon_{i}
\}.
\] 
Then the residue map of $F$ at the origin equals to $s$:
\[
\Res_{0}\ F
=
\int_{\Gamma}
\frac{dF_{1}\wedge\cdots \wedge dF_{n}}{F_{1}\cdots F_{n}}=s.
\]
\end{Theorem}
\begin{proof}
See \cite[p~60]{DAngelo-Book-1993},  \cite[p~666-667]{Griffiths-Harris-1994}, \cite[p~140, Proposition~1]{Chirka-1989} for discussion.
\end{proof}

\subsubsection{}\label{LAG-para-3-3} We will show that given Theorem \ref{LAG-thm-characterisations of complete intersections 2}, one has $\mult_{0}\ \Jac(F)\leqslant s-1$.

\begin{Theorem}\label{LAG-thm-Ideal membership by Residues}
Let $h$ be a holomorphic function germ. If $h\in\mathcal{I}_{F}$, then 
\[
\int_{\Gamma}\frac{h\ dz_{1}\wedge\cdots\wedge dz_{n}}{F_{1}\cdots F_{n}}=0.
\]
\end{Theorem}
\begin{proof}
See \cite[p~64]{DAngelo-Book-1993}.
\end{proof}

\begin{Corollary}\label{LAG-cor-Jacobian does not belong to original ideal}
Let $F_{1}$, \dots, $F_{n}$ be holomorphic function germs in $\mathcal{O}_{\mathbb{C}^{n},0}$ vanishing at the origin, whose varieties they define have complete intersection at the origin. Let $F$ be the map in Theorem \ref{LAG-thm-characterisations of complete intersections 2}(ii). Then $\Jac(F)\notin\langle F_{1},\dots, F_{\NN}\rangle$.
\end{Corollary}
\begin{proof}
By Theorem \ref{LAG-thm-characterisations of complete intersections 2}(iii),
\[
0\neq s=
\int_{\Gamma}
\frac{dF_{1}\wedge \cdots \wedge dF_{n}}{F_{1}\cdots F_{n}}
=
\int_{\Gamma}
\frac{\Jac(F)\ dz_{1}\wedge\cdots\wedge dz_{n}}{F_{1}\cdots F_{n}}.
\]
Hence $\Jac(F)\notin\langle F_{1},\dots, F_{n}\rangle$ by Theorem \ref{LAG-thm-Ideal membership by Residues}.
\end{proof}

\begin{Corollary}\label{LAG-cor-multiplicity of Jacobian is always lesser}
Let $F_{1}$,\dots, $F_{n}$ be holomorphic function germs in $\mathcal{O}_{\mathbb{C}^{n},0}$ vanishing at the origin so that the ideal $\mathcal{I}_{F}$ has finite intersection multiplicity with data $(p,q,s)$. Then the multiplicity of $\Jac(F)$ at the origin cannot be greater than or equal to $s$.
\end{Corollary}
\begin{proof}
Suppose otherwise that $\mult_{0}\ \Jac(F)\geqslant s$, then $\Jac(F)\in\mathfrak{m}^{s}$. From the inequality $q\leqslant s$, there is an inclusion of ideals  $\mathfrak{m}^{s}\subseteq \mathfrak{m}^{q}\subseteq\langle F_{1},\dots,F_{n}\rangle$. Hence $\Jac(F)\in\langle F_{1},\dots,F_{n}\rangle$, which contradicts Corollary \ref{LAG-cor-Jacobian does not belong to original ideal}.
\end{proof}

\subsubsection{Miscellaneous Result} We will state the following result which will be used later.

\begin{Proposition}\label{LAG-prop-intersectionnumberofproductfg=sumofintersection}
Let $f_{1},\dots,f_{n-1}$, $f$, and $g$ be holomorphic function germs in $\mathcal{O}_{\mathbb{C}^{n},0}$ such that 
\[
\dim_{\mathbb{C}}\ 
\mathcal{O}_{\mathbb{C}^{n},0}
/
\langle f_{1},\dots,f_{n-1},fg\rangle
< \infty.
\]
Then
\[
\dim_{\mathbb{C}}\ 
\mathcal{O}_{\mathbb{C}^{n},0}
/
\langle f_{1},\dots,f_{n-1},fg\rangle
=
\dim_{\mathbb{C}}\ 
\mathcal{O}_{\mathbb{C}^{n},0}
/
\langle f_{1},\dots,f_{n-1},f\rangle
+
\dim_{\mathbb{C}}\ 
\mathcal{O}_{\mathbb{C}^{n},0}
/
\langle f_{1},\dots,f_{n-1},g\rangle.
\]
\end{Proposition}

\begin{proof}
See \cite[p~60, Theorem~1]{DAngelo-Book-1993}
\end{proof}

\section{Ideals Generated by the Components of Gradient}
\subsubsection{} In this section we shall study the ideals generated by the components of the gradient of a holomorphic function. Let $f\in\mathcal{O}_{\mathbb{C}^{n},0}$ be a holomorphic function germ such that $f(0)=0$. In a first moment, it will be shown that there exists a positive integer $k$ with
\[
f^{k}
\in 
\left\langle
\frac{\partial f}{\partial z_{1}},
\dots,
\frac{\partial f}{\partial z_{n}}
\right\rangle.
\]
In a second moment, more accurately, it will be shown that $k=\dim\ \mathbb{C}^{n}=n$ works (optimally) rendering $k$ effective.
\subsubsection{Example} In $1$-dimensional complex analysis, every holomorphic function $f(\zeta)$ with $f(0)=0$ may be factorised as
\[
f(\zeta)=\zeta^{k}g(\zeta),
\]
where $g(0)\neq 0$. A differentiation yields
\[
f'(\zeta)=\zeta^{k-1}(kg(\zeta)+\zeta g'(\zeta)),
\]
and hence $f\in\langle f'\rangle$. Therefore, $k=1$ works in this case. In the next few paragraphs we will recall some notions in algebraic geometry.

\subsubsection{Spec, Zariski Topology} Let $A$ be a commutative ring with $1$. We let 
\[
\Spec\ A:=\{
\mathfrak{p}\subset A:\ \mathfrak{p}
\text{ is a prime ideal in }A\}.
\]
For every ideal $I\subset A$, set
\[
V_{A}(I)
=
\{\mathfrak{p}\in\Spec A:\ 
I\subseteq \mathfrak{p}\}.
\]
The sets $V_{A}(I)$ are defined as closed sets in $\Spec\ A$, and the collection
\[
\{V_{A}(I):\ I\text{ is an ideal of }A\}
\]
defines the {\sl Zariski topology} of $\Spec\ A$. For principal ideals $\langle f\rangle$, $V_{A}(\langle f\rangle)$ may be written as $V_{A}(f)$. Therefore, 
\[
D_{A}(f):=\Spec\ A\backslash V_{A}(f)
=
\{\mathfrak{p}\in\Spec\ A:\ f\notin \mathfrak{p}\}
\]
is open in $\Spec\ A$. The collection
\[
\{D_{A}(f):\ f\in A\}
\]
forms a basis for the open set in the Zariski topology. To see this, for any ideal $I$, one has
\[
\Spec\ A\backslash V_{A}(I)
=
\bigcup_{f\in I}D_{A}(f).
\]

\subsubsection{}\label{idealgrad-para-4} For any ideal $I\subset A$, there is a one-to-one correspondance between $\Spec\ A/I$ and $V_{A}(I)$. On the other hand, let $f\in A$ and $A_{f}$ be its localisation. Every element in $A_{f}$ is a class with representative $a/f^{k}$ for some $a\in A$ and $k\in\mathbb{N}$. The two representatives $a/f^{k}$ and $b/f^{l}$ are equal if there exists $j\geqslant 0$ such that
\[
f^{j}(af^{l}-bf^{k})=0
\eqno
{\scriptstyle{(\text{in }A})}.
\]
It is easily seen that like $\mathbb{Q}$ the quotient numbers of $\mathbb{Z}$, $A_{f}$ has a ring structure, and there is a one-to-one correspondance (as sets) between $\Spec\ A_{f}$ and $D_{A}(f)$.

\subsubsection{} Recall that a commutative ring with $1$ is {\sl semi-local} if it has only finitely many maximal ideals. We state the  following Artin-Tate theorem.

\begin{Theorem}
Let $A$ be a Noetherian integral domain. Then $A$ is semi-local with $\dim\ A\leqslant 1$ if and only if there exists $f\in A$ such that $A_{f}$ is a field.
\end{Theorem}

\begin{proof}
See \cite{Gortz-Wedhorn-2010}, page 562, Corollary B62.
\end{proof}

\subsubsection{}\label{idealgrad-para-6} Recall that the {\sl Krull dimension} of $A$ is given by
\[
\dim\ A
:=
\sup\{k:\ 
\mathfrak{p}_{0}
\subsetneq
\mathfrak{p}_{1}
\subsetneq
\cdots
\subsetneq
\mathfrak{p}_{k}
\}.
\]
Moreover, if $A$ is local, Artin-Tate's theorem may be restated as follows: there exists $f\in A$ such that $A_{f}$ is a field if and only if $\dim\ A\leqslant 1$.

\subsubsection{} For any germ variety $V(\mathcal
{I})$ defined by an ideal $\mathcal{I}\subset \mathcal{O}_{\mathbb{C}^{n},0}$, the Krull dimension of $\mathcal{O}_{\mathbb{C}^{n},0}
/\mathcal{I}$ coincides with the usual intuition of dimension. 

\subsubsection{}\label{idealgrad-para-8} To see this, recall that the {\sl Weierstrass dimension} of a germ of complex space $(X,x)$ is the least number $k$ such that there exists a Noether normalisation $\pi^{*}:\mathcal{O}_{\mathbb{C}^{k},x}
\hookrightarrow \mathcal{O}_{X,x}$. Both the Weierstrass dimension of $(X,x)$ and the Krull dimension of $\mathcal{O}_{X,x}
:=\mathcal{O}_{\mathbb{C}^{n},x}
/\mathcal{I}(X,x)$ coincide (~ \cite[Theorem~4.1.9, pp~131]{deJong-Pfister-2000}). The Noether normalisation $\pi^{*}$ is uniquely induced by the projection 
\[
\pi:(X,x)\rightarrow (\mathbb{C}^{k},x)
\]
of the germ variety $(X,x)$ onto $(\mathbb{C}^{n},x)$ with finite fibres. By \cite[p~129, Lemma~4.14]{deJong-Pfister-2000}, $\dim\ (X,x)=\dim\ (\mathbb{C}^{k},x)=k$. Hence to say that $\dim\ \mathcal{O}_{X,x}\leqslant 1$ is to say that either $X$ is of dimension $1$ or $0$ (but the dimension need not be pure).

\subsubsection{} The following lemma is a restatement of the Artin-Tate's theorem in more geometric terms.

\begin{Lemma}\label{idealgrad-lem-geometric meaning of artin-tate theorem}
Let $A$ be an integral Noetherian ring, and let $(0)$ be a point in $\Spec\ A$. Then the set $\{(0)\}$ is open in the Zariski topology of $\Spec\ A$ if and only if $\Spec\ A$ is a finite set and $\dim\ A\leqslant 1$.
\end{Lemma}

\begin{proof}
First, it will be shown that the singleton $\{(0)\}$ is open in $\Spec\ A$ if and only if there exists $f\in A$ such that $A_{f}$ is a field. Then secondly, it will be shown that $\Spec\ A$ is a finite set and $\dim\ A\leqslant 1$ if and only if $A$ is semi-local (finitely many maximal ideals) and $\dim\ A\leqslant 1$.

For the first assertion, suppose that $\{(0)\}$ is open, then for some index $\Lambda$,
\[
\{(0)\}=
\bigcup_{i\in\Lambda}
D_{A}(f_{i}).
\]
Therefore, $(0)\in D_{A}(f_{i})$ for some $i\in \Lambda$ and hence 
\[
\{(0)\}
\subseteq
D_{A}(f_{i})
\subseteq 
\bigcup_{i\in\Lambda}D_{A}(f_{i})
=
\{(0)\},
\]
so that $\{(0)\}=D_{A}(f_{i})$. By paragraph \ref{idealgrad-para-4}, this means that the ring $A_{f}$ has only $(0)$ as its prime ideal, and so $A_{f}$ is a field. Conversely, if there exists $f\in A$ such that $A_{f}$ is a field, then $\Spec\ A_{f}=D_{A}(f)=\{(0)\}$ as sets. Consequently, the singleton $\{(0)\}$ is open in $\Spec\ A$.

For the second assertion, suppose that $\Spec\ A$ is a finite set and $\dim\ A\leqslant 1$. The first condition implies that $A$ only has finitely many prime ideals. On the other hand, for any maximal chain of prime ideals
\[
(0):=\mathfrak{p}_{0}
\subsetneq 
\mathfrak{p}_{1}
\subsetneq
\cdots
\subsetneq
\mathfrak{p}_{k}
\]
one has $k\leqslant 1$ by the second condition. If $k=1$, then $\mathfrak{p}_{k}$ is a maximal ideal. If $k=0$, then $\mathfrak{p}_{0}$ is just the zero ideal. Hence any non-zero prime ideal $\mathfrak{p}_{1}$ is maximal, and since $A$ has finitely many prime ideals, $A$ is semi-local. Conversely, suppose that $A$ is semi-local and $\dim\ A\leqslant 1$. By the chain of inclusion of prime ideals above, any non-zero prime ideal is maximal by second condition. Moreover, $A$ is semi-local means that there are only finitely many maximal ideals. Hence $\Spec\ A$ is a finite set and $\dim\ A\leqslant 1$.
\end{proof}

\subsubsection{} The lemma above is used to prove the following lemma.

\begin{Lemma}
Let $A$ be a local integral domain, $(0)\neq \mathfrak{m}$, and $\dim\ A=n\geqslant 1$. Let $f\in\ A\backslash \{0\}$ be a non-zero element with $f\in\mathfrak{m}$. Then there exists a prime ideal $\mathfrak{p}$ such that $\dim\ A/\mathfrak{p}=1$ and $f\notin \mathfrak{p}$.
\end{Lemma}

\begin{proof}
We will construct by induction on $k$ a sequence of inclusion of prime ideals
\[
(0):=\mathfrak{p}_{0}
\subsetneq
\mathfrak{p}_{1}
\subsetneq
\cdots
\subsetneq
\mathfrak{p}_{k}
\subsetneq
\mathfrak{m}
\]
with the following conditions:

\smallskip\noindent{\bf (i)}
$f\notin \mathfrak{p}_{k}$;

\smallskip\noindent{\bf (ii)}
for every $1\leqslant i\leqslant k$, the prime ideal $\mathfrak{p}_{i}$ is of height one over $\mathfrak{p}_{i-1}$. In other words, there is no prime ideal $\mathfrak{q}$ with $\mathfrak{p}_{i-1}\subsetneq \mathfrak{q}\subsetneq \mathfrak{p}_{i}$;

\smallskip\noindent{\bf (iii)}
either $\dim\ A/\mathfrak{p}_{k}=1$ or $\dim\ A/\mathfrak{p}_{k}\leqslant 
(\dim\ A)-k$.

Suppose such a $\mathfrak{p}_{k}$ is constructed, we see that the second condition in (iii) is always a consequence of (ii). This is because  from the definition of height of a prime, one always has $k\leqslant \Ht(\mathfrak{p}_{k})$. By \cite[p~133, Remark~4.1.15]{deJong-Pfister-2000}, one always has
\[
\Ht(\mathfrak{p})+
\dim(A/\mathfrak{p}_{k})
\leqslant \dim\ A.
\]
Therefore,
\[
k+\dim(A/\mathfrak{p}_{k})
\leqslant 
\Ht(\mathfrak{p})+
\dim(A/\mathfrak{p}_{k})
\leqslant \dim\ A,
\]
and consequently $\dim(A/\mathfrak{p}_{k})\leqslant (\dim\ A)-k$.

 \textbf{\textsc{Step 1}:}  We let $A_{0}:=A/\mathfrak{p}_{0}=A$. By hypothesis that $f\neq 0$ in $A$, therefore $(0)\in D_{A}(f)$, and so the open set $D_{A}(f)$ is non-empty. There are two cases according to whether $D_{A}(f)$ has exactly one or more than one elements:  

\smallskip\noindent{\bf (1)} 
Suppose that  $D_{A}(f)=\{(0)\}$, then set $\{(0)\}$ is Zariski open in $\Spec\ A$. By Lemma \ref{idealgrad-lem-geometric meaning of artin-tate theorem} and remark in paragraph \ref{idealgrad-para-6}, one has $\dim\ A\leqslant 1$. By the hypothesis that $(0)\neq \mathfrak{m}$, hence $\dim\ A=1$ and the proof is finished.

\smallskip\noindent{\bf (2)} Otherwise, the Zariski open set $D_{A}(f)$ contains another prime ideal $\mathfrak{p}_{1}'$ in $A$ such that $\mathfrak{p}_{1}'\neq (0)$. Moreover, since $\dim\ A$ is finite, $\Ht(\mathfrak{p}')$ is finite, say $h_{1}$. Consider a  maximal chain of prime ideals
\[
(0)=\mathfrak{p}_{0}
\subsetneq
\mathfrak{q}_{1,1}
\subsetneq
\cdots
\subsetneq
\mathfrak{q}_{1,h_{1}}
=\mathfrak{p}_{1}'
\]
whose length is  $\Ht(\mathfrak{p}_{1}')$. By maximality of the chain, $\Ht(\mathfrak{q}_{1,1})=1$. Moreover, $f\notin \mathfrak{p}_{1}'$ implies that $f\notin\mathfrak{q}_{1,1}$. Therefore, we may let $\mathfrak{p}_{1}:=\mathfrak{q}_{1,1}$.

 \textbf{\textsc{Inductive Step}:} Once the prime ideals $\mathfrak{p}_{1}$, \dots, $\mathfrak{p}_{k}$ have been constructed, the existence of $\mathfrak{p}_{k+1}$ that satisfies the first two conditions (i) and (ii) will be constructed. The idea is to pass through the quotient $\pi_{k}:A\rightarrow A/\mathfrak{p}_{k}:=A_{k}$ and repeat the steps as in  \textbf{\textsc{Step 1} }. This time, since $f\in \mathfrak{m}$ but $f\notin\mathfrak{p}_{k}$, hence $\mathfrak{p}_{k}\neq\mathfrak{m}$. So in $A_{k}$, the zero ideal $(0)$ is not maximal. Also, the hypothesis that $f\notin \mathfrak{p}_{k}$ implies that the class $f$ is not zero in $A_{k}$. Therefore, the open set $D_{A_{k}}(f)$ is not empty since it contains the zero ideal $(0)$. We study the open set $D_{A}(f)$ with the two following cases just as before:
 
 \smallskip\noindent{\bf (1)}
 Either $D_{A_{k}}(f)$ is a singleton, meaning $D_{A_{k}}(f)=\{(0)\}$. Then $\{(0)\}$ is open in $\Spec\ A_{k}$. By Lemma \ref{idealgrad-lem-geometric meaning of artin-tate theorem}, one has $\dim\ A_{k}\leqslant 1$. But since $(0)\neq \mathfrak{m}$, $\dim\ A_{k}=1$.
 
  \smallskip\noindent{\bf (2)}
Otherwise, in $A_{k}$ we may find a prime ideal $\mathfrak{P}_{k+1}\subseteq A_{k}$ such that $\mathfrak{P}_{k}\in D_{A_{k}}(f)$ and $\mathfrak{P}_{k+1}\neq (0)$. Since $\dim\ A_{k}$ is finite, so is $\Ht(\mathfrak{P}_{k+1})$, which is for example $h_{k}$. We let 
\[
(0):=\mathfrak{P}_{0}
\subsetneq
\mathfrak{Q}_{k+1,1}
\subsetneq
\cdots
\subsetneq 
\mathfrak{Q}_{h_{k+1},1}
:=
\mathfrak{P}_{k+1}
\]
be a maximal chain of prime ideals in $A_{k}$ corresponding whose length corresponds to the height of $\mathfrak{P}_{k+1}$. Therefore $\mathfrak{Q}_{k+1,1}$ is of height one by the maximality of the chain, and therefore we may let $\mathfrak{p}_{k+1}=\pi^{-1}(\mathfrak{Q}_{k+1,1})$. Moreover, since the class $f$ does not belong to $\mathfrak{P}_{k+1}$, the element $f$ does not belong to $\mathfrak{p}_{k+1}$. Hence $\mathfrak{p}_{k+1}$ satisfies the two conditions and the inductive step is complete.

To conclude, since $\dim\ A=n$, the length of the chain $k$ is at most $n-1$. In this case one has
\[
\dim\ A/\mathfrak{p}_{n-1}\leqslant 1.
\]
But since $f\in\mathfrak{m}$ and $f\notin\mathfrak{p}_{n-1}$, so $\dim\ A/\mathfrak{p}_{n-1}=1$, and the proof is complete.
\end{proof}

\subsubsection{}

\begin{Lemma}
Let $\mathcal{I}\subseteq
\mathcal{O}_{\mathbb{C}^{n},0}$ be an ideal and $Y:=V(\mathcal{I})$ be the variety defined by $\mathcal{I}$. Let $f$ be a holomorphic function germ vanishing at the origin and $X=V(\mathcal{I}+(f))$ be the variety defined by $\mathcal{I}+(f)$. If the inclusion $X\subsetneq Y$ is strict, then there exists an irreducible curve $C\subseteq Y$ passing through the origin such that $f|_{C}\not\equiv 0$.
\end{Lemma}

\begin{proof}
Almost half of the proof is done by the previous lemma. It suffices to observe that if the inclusion $X\subsetneq Y$ is strict, then there exists an irreducible component $W\subseteq Y$ passing through the origin on which $f|_{W}\not\equiv 0$. Once this is proved, then the condition that $f|_{W}\not\equiv 0$ implies that $f\notin\mathcal{I}(W,0)$. Since $W$ is irreducible, the ideal  $\mathcal{I}(W,0)$ is prime. Moreover, the dimension of $W$ is greater than $1$, otherwise $W$ will just be the origin but $f(0)=0$, which contradicts that $f|_{W}\not\equiv 0$. Put differently, one has $f$ a non-zero element in $\mathcal{O}_{W,0}
=\mathcal{O}_{\mathbb{C}^{n},0}
/\mathcal{I}(W,0)$, with $f\in\mathfrak{m}$. Moreover, $\mathcal{O}_{W,0}$ is a local integral domain with $(0)\neq \mathfrak{m}$ since $\dim\ W\geqslant 1$. By previous lemma, there exists a prime ideal $\mathfrak{P}$ in $\mathcal{O}_{W,0}$ such that the class $f$ does not lie in $\mathfrak{P}$ and $\dim\ \mathcal{O}_{W,0}/\mathfrak{P}=1$. If we let $\pi:\mathcal{O}_{\mathbb{C}^{n},0}\rightarrow 
\mathcal{O}_{W,0}$ be the usual ring homomorphism by quotient, one has the prime ideal $\mathfrak{p}:=\pi^{-1}\mathfrak{P}$ such that $\dim
\mathcal{O}_{\mathbb{C}^{n},0}
/\mathfrak{p}=\dim\ \mathcal{O}_{W,0}
/\mathfrak{P}=1$, with $f\notin\mathfrak{p}$. Hence  $C=V(\mathfrak{p})$ is an irreducible one-dimensional variety on which $f$ does not totally vanish.

It remains to prove the observation. By Primary decomposition theorem, the ideal $\mathcal{I}$ may be decomposed as an intersection 
\[
\mathcal{I}=
\bigcap_{i=1}^{M}
P_{i}
\]
of primary ideals $P_{i}$, whose radical $\sqrt{P_{i}}:=\mathfrak{p}_{i}$ is prime. Hence
\[
Y=V(\mathcal{I})=\bigcup_{i=1}^{M}V(P_{i})
=
\bigcup_{i=1}^{M}V(\mathfrak{p}_{i})
\]
where the last equality follows from Nullstellensatz.

We claim that there exists $i$ such that $f|_{V(\mathfrak{p}_{i})}\not\equiv 0$. Otherwise, for every $1\leqslant i\leqslant M$, one has $f\in \mathcal{I}(V(\mathfrak{p}_{i}),0)=\mathfrak{p}_{i}$ by Nullstellensatz. Therefore, there exists a positive integer $k$ such that for all $1\leqslant i\leqslant M$, one has $f^{k}\in P_{i}$, and hence $f^{k}\in \mathcal{I}$. Therefore,
\[
X=V(\mathcal{I}+(f))
=
V(\mathcal{I})\cap V(f)
=
V(\mathcal{I})\cap V(f^{k})
=
V(\mathcal{I}+(f^{k}))
=
V(\mathcal{I})
=
Y,
\]
which contradicts our assumption that $X\subsetneq Y$. The proof is complete.
\end{proof}

\subsubsection{} We have therefore arrived at one of the main results of this section.

\begin{Proposition}
Let $f\in\mathcal{O}_{\mathbb{C}^{n},0}$ with $f(0)=0$. Then there exists $k\in\mathbb{N}$ such that
\[
f^{k}\in 
\left\langle
\frac{\partial f}{\partial z_{1}},
\dots,
\frac{\partial f}{\partial z_{n}}
\right\rangle.
\]
\end{Proposition}
\begin{proof}
Let $X=V(\partial_{z_{1}}f,\dots,
\partial_{z_{n}}f,f)$ and $Y=V(\partial_{z_{1}}f,\dots,
\partial_{z_{n}}f)$. If $X\subsetneq Y$, by previous proposition, there exists an irreducible curve $C\subseteq Y$ such that $f|_{C}\not\equiv 0$. Let 
\begin{eqnarray*}
h:(\mathbb{C},0) &\longrightarrow & (C,0)\\
\zeta &\longmapsto & (h_{1}(\zeta),\dots,h_{n}(\zeta))
\end{eqnarray*}
be a local normalisation of $C$. Hence
\[
\frac{d f\circ h(\zeta)}{d\zeta}
=
\sum_{k=1}^{n}
\frac{\partial f}{\partial z_{k}}
\circ h(\zeta)
h_{k}'(\zeta).
\]
Since $h(\mathbb{C},0)\subseteq C$,  $(\partial_{z_{k}}f)\circ h(\zeta)\equiv 0$ for all $1\leqslant k\leqslant n$. Therefore, $\frac{d}{d\zeta}(f\circ h)\equiv 0$, and hence $f\circ h$ is a constant on $C$. Moreover, since $f\circ h(0)=0$,  $f\circ h$ vanishes identically on $C$, which contradicts our hypothesis that $f|_{C}\not\equiv 0$. Therefore, $X=Y$ and by Nullstellensatz, there exists an integer $k$ such that $f^{k}\in \mathcal{I}$ and the proof is complete.
\end{proof}

\subsection{Ideals Generated by Components of Gradients: Effective Aspects}

\subsubsection{}\label{idealgrad-para-1-1} In fact, the exponent $k$ in the previous proposition may be taken to be $k=n=\dim\ \mathbb{C}^{n}$. We summarise some of the details in \cite[p~59]{Demailly-1996}.

\subsubsection{}\label{idealgrad-para-1-2}
\begin{Definition}\label{idealgrad-def-integral closure of I}
Let $\mathcal{I}$ be an ideal in $\mathcal{O}_{\mathbb{C}^{n},0}$. The {\sl integral closure} of $\mathcal{I}$, denoted by $\bar{\mathcal{I}}$, is the set of germs $u\in\mathcal{O}_{\mathbb{C}^{n},0}$ such that there exist $d\in\mathbb{N}_{\geqslant 1}$ and $\alpha_{s}\in\mathcal{I}^{s}$ for $1\leqslant s\leqslant d$ with: 
\[
u^{d}+a_{1}u^{d-1}
+
\dots
+
a_{d}=0.
\]
\end{Definition}

\begin{Definition}\label{idealgrad-def-I^(k)}
Let $\mathcal{I}=\langle F_{1},\dots,F_{\NN}\rangle$ be an ideal of $\mathcal{O}_{\mathbb{C}^{n},0}$ generated by $\NN$ elements. Let $u\in\mathbb{R}_{+}$. The ideal $\overline{\mathcal{I}}^{(k)}$ is defined by
\[
\overline{\mathcal{I}}^{(k)}
=
\{u\in\mathcal{O}_{\mathbb{C}^{n},0}:\ 
|u|\leqslant C|F|^{k}\}
\]
for some constants $C\geqslant 0$, and where $|F|^{2}=|F_{1}|^{2}+\dots+|F_{\NN}|^{2}$.
\end{Definition}

\subsubsection{}\label{idealgrad-para-1-3} By \cite[p~60, Proposition~12.2]{Demailly-1996}, for every $k$, $l$ positive real numbers, $\mathcal{I}^{(k)}\cdot \mathcal{I}^{(l)}\subseteq \mathcal{I}^{(k+l)}$. Moreover, $\overline{\mathcal{I}}^{(1)}=
\overline{\mathcal{I}}$ which is the integral closure of the ideal $\mathcal{I}$ (\cite[p~61, Corollary~12.5]{Demailly-1996}).

\subsubsection{}\label{idealgrad-para-1-4} By the Brian\c{c}on-Skoda theorem (1974), if $p=\min\{n-1,\NN-1\}$, then $\overline{\mathcal{I}}^{(k+p)}\subseteq
\mathcal{I}^{k}$ for all $k\in\mathbb{N}$.

\subsubsection{}\label{idealgrad-para-1-5} Let $f\in\mathcal{O}_{\mathbb{C}^{n},0}$ be a holomorphic function germ at the origin, and let $\mathcal{I}_{f}$ and $J(f)$ denote the following ideals:
\begin{eqnarray*}
\mathcal{I}_{f}
&:=&
\left\langle
z_{1}\frac{\partial f}{\partial z_{1}},\dots,
z_{n}\frac{\partial f}{\partial z_{n}}
\right\rangle,\\
J(f)
&:=&
\left\langle
\frac{\partial f}{\partial z_{1}},\dots,
\frac{\partial f}{\partial z_{n}}
\right\rangle.
\end{eqnarray*}
It is evident that $\mathcal{I}_{f}\subset J(f)$. Moreover, one has $f\in\overline{\mathcal{I}_{f}}=
\mathcal{I}^{(1)}$ (\cite[p~62, Corollary~12.6]{Demailly-1996}). Therefore, by Paragraph \ref{idealgrad-para-1-3}, $f^{k+n-1}\in\mathcal{I}^{(k+n-1)}$. By Brian\c{c}on-Skoda theorem, for all $k\in\mathbb{N}_{\geqslant 1}$
\[
f^{k+n-1}
\in 
\mathcal{I}_{f}^{k}.
\]
By setting $k=1$, $f^{n}\in J(f)$. This completes the proof of the following proposition:

\begin{Proposition}\label{idealgrad-prop-f^n lies in the gradient}
Let $f\in\mathcal{O}_{\mathbb{C}^{n},0}$ with $f(0)=0$. Then
\[
f^{n}\in 
\left\langle
\frac{\partial f}{\partial z_{1}},
\dots,
\frac{\partial f}{\partial z_{n}}
\right\rangle.
\]
\end{Proposition}

\subsection{Application}

\subsubsection{}\label{idealgrad-para-2-1} Let $F_{1}$,\dots,$F_{\NN}$ be holomorphic function germs in $\mathcal{O}_{\mathbb{C}^{n},0}$ such that intersection multiplicity of the ideal $\langle F_{1},\dots, F_{\NN}\rangle$ is finite with data $(p,q,s)$. We will show that the ideal
\[
\left\langle 
\frac{\partial F_{i}}{\partial z_{j}}:\ 
1\leqslant i\leqslant \NN,\ 1\leqslant j\leqslant n
\right\rangle
\]
has an effectively bounded intersection multiplicity. More precisely,

\begin{Proposition}\label{idealgrad-prop-ideal of gradient has finite intersection multiplicity}
For any $n\in\mathbb{N}_{\geqslant 1}$,let $F_{1}$,\dots,$F_{\NN}$ be holomorphic function germs in $\mathcal{O}_{\mathbb{C}^{n},0}$ vanishing at the origin such that the ideal $\langle F_{1},\dots,F_{\NN}\rangle$ has finite intersection multiplicity with data $(p,q,s)$. Then 
\[
\dim_{\mathbb{C}}\ 
\mathcal{O}_{\mathbb{C}^{n},0}
\bigg/
\left\langle
\frac{\partial F_{i}}{\partial z_{j}}
:\ 
1\leqslant i\leqslant \NN,\ 
1\leqslant j\leqslant n
\right\rangle
\leqslant
\left(
\begin{matrix}
(n^{2}+2n)s+n-1\\
(n^{2}+2n)s-1
\end{matrix}
\right).
\]
\end{Proposition}
\begin{proof}
By Proposition \ref{idealgrad-prop-f^n lies in the gradient}, for each $1\leqslant i\leqslant \NN$,
\[
F_{i}^{n}
\in 
\left\langle
\frac{\partial F_{i}}{\partial z_{1}},
\dots,
\frac{\partial F_{i}}{\partial z_{n}}
\right\rangle.
\]
Evidently,
\[
\langle F_{1}^{n},\dots,F_{\NN}^{n}\rangle
\subseteq
\left\langle
\frac{\partial F_{i}}{\partial z_{j}}:\ 
1\leqslant i\leqslant \NN,\ 
1\leqslant j\leqslant n
\right\rangle.
\]
As a result,
\[
\dim_{\mathbb{C}}\ 
\mathcal{O}_{\mathbb{C}^{n},0}
\bigg/
\left\langle
\frac{\partial F_{i}}{\partial z_{j}}:\ 
1\leqslant i\leqslant \NN,\ 
1\leqslant j\leqslant n
\right\rangle
\leqslant 
\dim_{\mathbb{C}}\ 
\mathcal{O}_{\mathbb{C}^{n},0}
\big/
\langle F_{1}^{n}
,\dots,
F_{\NN}^{n}
\rangle.
\]
It suffices to estimate the term on the right. By the hypothesis that
\[
|z|^{p}\lesssim
\sum_{i=1}^{\NN}|F_{i}|,
\]
Jensen's inequality yields
\begin{eqnarray*}
|z^{np}|
&\lesssim &
\left(\sum_{i=1}^{\NN}|F_{i}|\right)^{n}
\\
&=&
\NN^{n}\left(
\sum_{i=1}^{\NN}
\frac{1}{\NN}|F_{i}|
\right)^{n}\\
&\leqslant &
\NN^{n}\sum_{i=1}^{\NN}\frac{1}{\NN}|F_{i}|^{n}\\
&=&
\NN^{n-1}\sum_{i=1}^{\NN}|F_{i}|^{n}
\lesssim 
\sum_{i=1}^{\NN}|F_{i}^{n}|.
\end{eqnarray*}
By Theorem \ref{LAG-thm-characterisations of complete intersection 1}, the ideal $\langle F_{1}^{n},\dots,F_{\NN}^{n}\rangle$ has finite intersection multiplicity with data $(p',q',s')$. By the definition of $p'$, one has $p'\leqslant np$. At the same time, by Proposition \ref{LAG-prop-Relations between the intersection invariants},
\[
s'\leqslant
\binom{
n+q'+1}{
q'-1}.
\]
Also by Proposition \ref{LAG-prop-Relations between the intersection invariants}, using $p'\leqslant np$, the $q'$ in the inequality above has a bound
\[
q'\leqslant (n+2)p'\leqslant (n+2)np\leqslant (n+2)nq\leqslant (n+2)ns.
\]
Hence
\[
s'\leqslant
\binom{
n+(n+2)ns-1}{
(n+2)ns-1}
\]
and the proof is complete.
\end{proof}
\subsubsection{Remark}\label{idealgrad-para-2-2} The estimate can be made more precise in $\NN=n$. In this case by repeated application of Proposition \ref{LAG-prop-intersectionnumberofproductfg=sumofintersection}
,
\[
\dim_{\mathbb{C}}\ 
\mathcal{O}_{\mathbb{C}^{n},0}
\big/
\langle F_{1}^{n},\dots, F_{n}^{n}\rangle
=
n^{n}\dim_{\mathbb{C}}\ 
\mathcal{O}_{\mathbb{C}^{n},0}
\big/
\langle F_{1},\dots,F_{n}\rangle
=
n^{n}s.
\]
Therefore,
\[
\dim_{\mathbb{C}}\ 
\mathcal{O}_{\mathbb{C}^{n},0}
\bigg/
\left\langle
\frac{\partial F_{i}}{\partial z_{j}}:\ 
1\leqslant i,j\leqslant n
\right\rangle
\leqslant n^{n}s.
\]

\section{Multiplicity of an Ideal}

\subsubsection{}\label{multofideal-ss1} Following \cite{Chirka-1989}, we will present the notion of  multiplicity of an ideal of holomorphic functions defining a pure dimensional variety. Let $F\in\mathcal{O}_{\mathbb{C}^{n},0}$ be a holomorphic function germ. We may write $F$ as an infinite sum of homogeneous polynomials
\[
F=\sum_{k=m}^{\infty}F_{k},
\]
where each $F_{k}$ is of degree $k$, with $F_{m}\not\equiv 0$.The multiplicity of $F$ at $0$ is then equal to $m$.

\subsubsection{}\label{multofideal-ss2} Another way to characterise multiplicity is to look at the order of vanishing of $F$ along generic lines. Indeed, let
\begin{eqnarray*}
\varphi:\mathbb{C} 
&\longrightarrow & 
\mathbb{C}^{n}\\
\zeta &\longmapsto & (c_{1}\zeta,\dots,c_{n}\zeta)
\end{eqnarray*}
be a parametrisation of a line. Composing $\varphi$ with $F$ gives
\[
F\circ\varphi
=
\sum_{k=m}^{\infty}
F_{k}(c_{1}\zeta,\dots,c_{n}\zeta)
=
\sum_{k=m}^{\infty}
\zeta^{k}F_{k}(c_{1},\dots,c_{n}).
\]
Therefore, any vector $(c_{1},\dots,c_{n})$ satisfying 
\[
F_{m}(c_{1},\dots,c_{n})\neq 0
\]
will imply that $\ord_{0}\ F\circ \varphi=m=
\mult_{0}\ F$. Since every line is a complete intersection of $n-1$ hyperplanes $H_{1}$,\dots,$H_{n-1}$,  the multiplicity of $F$ is the intersection multiplicity of $V(F)=\{z\in\mathbb{C}^{n}:F(z)=0\}$ with $n-1$ generic hyperplanes. In other words, if  $H_{i}=\{z\in\mathbb{C}^{n}:L_{i}=0\}$ for some linear function $L_{i}$, then
\[
\mult_{0}F
=
\dim_{\mathbb{C}}
\mathcal{O}_{\mathbb{C}^{n},0}
\big/
\langle 
F,L_{1},\dots,L_{n-1}
\rangle.
\]

\subsubsection{}\label{multofideal-ss3} More generally at $0$, for $1\leqslant q\leqslant n-1$, let $\mathcal{I}_{F}=\langle F_{1},\dots,F_{q}\rangle$ be the ideal generated by $q$ holomorphic function germs and assume that it forms a {\sl regular sequence}\footnote{Let $R$ be a local ring. A sequence of non-units $f_{1}$,\dots,$f_{k}$ is called a regular sequence if for all $1\leqslant i\leqslant k$, the class $f_{i}$ is not a  zero divisor of $R/\langle f_{1},\dots,f_{i-1}\rangle$}. We would like to find a positive integer $m$ analogous to the multiplicity of a function such that for $(n-q)$ generic hyperplanes,
\[
m=
\dim_{\mathbb{C}}
\mathcal{O}_{\mathbb{C}^{n},0}
\big/
\langle F_{1},\dots,F_{q},L_{1},\dots,L_{n-q}\rangle.
\]
The important point about $F_{1}$,\dots,$F_{q}$ being a regular sequence is that the variety $V(\mathcal{I}_{F})$ defined by the ideal $\mathcal{I}_{F}$ is of {\sl pure} dimension $n-q$, due to the property of {\sl Cohen-Macaulayness}. This allows us to apply the results in \cite[Chapter~2]{Chirka-1989}.

\begin{Definition}[Tangent Cones, {\cite[p.~79]{Chirka-1989}}]
\label{multofideal-def: tangent cones}
Let $E$ be an arbitrary set in $\mathbb{R}^{n}$. A vector $v\in\mathbb{R}^{n}$ is called {\sl tangent} to $E$ at a point of the closure $a\in \bar{E}$ if there exist a sequence of points $a_{j}\in E$ and positive numbers $t_{j}>0$ such that $a_{j}\rightarrow a$ and 
\[
t_{j}(a_{j}-a)\rightarrow v
\eqno
{\scriptstyle{(j\,\rightarrow\,\infty}).}
\]
The set of all such tangent vectors at $a$ is denoted by $C(E,a)$, and is called the {\sl tangent cone to $E$ at the point $a$}.
\end{Definition}

\subsubsection{}\label{multofideal-ss4} The set $C(E,a)$ is a cone with vertex $0$: if $v\in C(E,a)$, then  the vectors $tv$ lie in $C(E,a)$ for all $t>0$. Geometrically, the cone is a set of limit positions of secants of $E$ passing through $a$.

\subsubsection{}\label{multofideal-ss5} If $V$ is a pure one-dimensional analytic set in $\mathbb{C}^{n}$, the tangent cone at any $a\in V$ is a finite union of complex lines (\cite[p.~80, Corollary]{Chirka-1989}).

\subsubsection{}\label{multofideal-ss6} In general, if $0\in V$ is a pure analytic subset of a domain $D$ in $\mathbb{C}^{n}$, then $C(V,0)$ is a pure $p$-dimensional algebraic set in $\mathbb{C}^{n}$ (c.f. \cite[p.~83, Corollary]{Chirka-1989}).

\subsubsection{}\label{multofideal-ss7} We recall that if a $n-p$-dimensional variety $V$ is defined by $p$ holomorphic functions $F_{1}$,\dots, $F_{p}$ which form a regular sequence, then $V$ is a pure $(n-p)$-dimensional analytic variety.

\subsection{Multiplicities of Analytic Sets}

We refer the readers to \cite[p.~120]{Chirka-1989} for more details.

\subsubsection{}\label{multofideal-ss8} Let $V$ be a pure $p$ dimensional analytic set in $\mathbb{C}^{n}$, and let $a\in V$. Let $L$ be an $(n-p)$-dimensional complex subspace in $\mathbb{C}^{n}$, such that $a$ is an isolated point of the set $V\cap (a+L)$. Then there is a domain $U\ni a$ in $\mathbb{C}^{n}$ of the form $U=U'\times U''\subseteq \mathbb{C}^{n-p}\times\mathbb{C}^{p}$ such that $V\cap U\cap (a+L)=\{a\}$, and the projection
\[
\pi_{L}:V\cap U\rightarrow U'\subseteq L^{\perp}
\]
along $L$ is a ramified $k$-sheeted  analytic cover. This number $k$ is the multiplicity of the projection $\pi_{L}|_{V}$ at $a$, denoted by $\mu_{a}(\pi_{L}|_{V})$.

\subsubsection{}\label{multofideal-ss9} For simplicity, suppose that $0\in V$ and $a=0$ in the previous paragraph, the multiplicity of intersection of $V$ with $L$ is $\mu_{0}(\pi_{L}|_{V})$. See \cite[p~139, Corollary]{Chirka-1989} and \cite[p~140, Proposition~1]{Chirka-1989}.

\begin{Definition}[Multiplicity of an Analytic Set at a point]
\label{multofideal-def: multiplicity of an analytic set at a point}
Let $V$ be a pure $p$-dimensional analytic set in $\mathbb{C}^{n}$ and let $a\in V$. For every $(n-p)$-dimensional plane $L$ which contains the origin such that 
\[
V\cap (a+L)=\{a\},
\]
the multiplicity of the projection $\mu_{a}(\pi_{L}|_{V})$ is finite. The multiplicity of $V$ at $a$ is given by 
\[
\mu_{a}(V)
:=
\min\{\mu_{a}(\pi_{L}|_{V}):
L\in G(n-p,n)\}.
\]
\end{Definition}

\begin{Example}
Suppose $V=\{F=0\}$ is a principal analytic set in a neighbourhood of $0\in\mathbb{C}^{n}$, and $F$ is the minimum defining function for $V$. Write $F$
\[
F=\sum_{k=\ordsmall_{0}F}^{\infty}F_{k}
\]
as a sum of homogeneous polynomials $F_{k}$ of degree $k$. Then by \cite[p.~83, Proposition~1]{Chirka-1989},
\[
C(V,0)=\{F_{\ordsmall_{0}F}=0\}.
\]
For any complex {\sl line} $L$ containing $0$, by \cite[p.~121, Proposition~1]{Chirka-1989},
\[
\mu_{0}(\pi_{L}|_{V})=
\ord_{0}F|_{L}
\geqslant 
\ord_{0}F,
\]
with equality if and only if $L\cap C(V,0)=\{0\}$. In other words, if 
\[
\zeta \mapsto (c_{1}\zeta,\dots, c_{n}\zeta)
\]
is a parametrisation of the line $L$, then the line has trivial intersection with $C(V,0)$ if and only if 
\[
F_{\ordsmall_{0}F}(c_{1},\dots,c_{n})\neq 0.
\]
This agrees with our intuition in paragraph \ref{multofideal-ss2}.
\end{Example}

\subsubsection{}\label{multofideal-ss10} More generally, 

\begin{Proposition}\label{multofideal-prop-critieriaforequality}
Let $V$ be a pure $p$-dimensional analytic set in a neighbourhood of $0\in\mathbb{C}^{n}$, and let $L\in G(n-p,n)$. The equality $\mu_{0}(\pi_{L}|_{V})=\mu_{0}$ holds if and only if the plane $L$ is transversal to $V$ at $0$. In other words,
\[L\cap C(V,0)=\{0\}.\]
\end{Proposition}
\begin{proof}
See \cite[p.~122, Proposition~2]{Chirka-1989}.
\end{proof}

\subsubsection{}\label{multofideal-ss11} Combining paragraphs \ref{multofideal-ss6}, \ref{multofideal-ss7}, \ref{multofideal-ss9}, and Proposition \ref{multofideal-prop-critieriaforequality}, we obtain

\begin{Proposition}\label{multofideal-prop-multiplicity-of-an-ideal}
Let $F_{1}$,\dots, $F_{p}$ be holomorphic function germs at the origin so that $F_{1}(0)=\cdots=F_{p}(0)=0$. Suppose that the sequence $F_{1}$,\dots,$F_{p}$ is regular so that the variety defined by the intersection $V:=V(F_{1},\dots,F_{p})$ is a pure $(n-p)$-dimensional analytic variety. Then there exists an integer $\mu_{0}(V)$ such that for a generic choice of $p$ hyperplanes given by the zeros of $n-p$ linear functions $L_{1},\dots,L_{n-p}$, one has
\[
\dim_{\mathbb{C}}
\mathcal{O}_{\mathbb{C}^{n},0}
\big/
\langle F_{1},\dots,F_{p},L_{1},\dots,L_{n-p}\rangle
=
\mu_{0}(A).
\]
\end{Proposition}

\subsection{Multiplicity of an Ideal -- Case of a Curve}

\subsubsection{} In this section, we will discuss more in depth of Proposition \ref{multofideal-prop-multiplicity-of-an-ideal} in the case where $p=n-1$. In other words, the ideal $\mathcal{I}_{F}=\langle F_{1},\dots,F_{n-1}\rangle$ forms a regular sequence in $\mathcal{O}_{\mathbb{C}^{n},0}$ so that the variety $V(F_{1},\dots,F_{n-1})$ is a pure 1-dimensional analytic variety, which is a union
\[
V(F_{1},\dots,F_{n-1})
=
\bigcup_{k=1}^{\MM}Z_{k}
\]
of its irreducible components $Z_{k}$.
\subsubsection{} For $1\leqslant k\leqslant \MM$, since each $Z_{k}$ is an irreducible curve, there exists a parametrisation 
\begin{eqnarray*}
n_{k}: (\mathbb{C},0) &\rightarrow & (Z_{k},0)\\
\zeta &\mapsto & 
(\zeta^{\mu_{k1}}a_{k1}(\zeta),\dots,
\zeta^{\mu_{kn}}a_{kn}(\zeta)),
\end{eqnarray*}
where for all $1\leqslant j\leqslant n$, $a_{kj}(0)\neq 0$ (c.f. \cite[p~164, Theorem~4.4.8]{deJong-Pfister-2000}) and \cite[p~165, Theorem~4.4.10]{deJong-Pfister-2000}).

\begin{Theorem}\label{multofideal-thm: compute intersection numbers by normalisation}
There exist positive integers  $m_{1}$,\dots,$m_{\MM}$ such that for any holomorphic function $f$ with
\[
\dim_{\mathbb{C}}
\mathcal{O}_{\mathbb{C}^{n},0}
/
\langle 
F_{1},\dots,F_{n-1},
f
\rangle<\infty,
\]
the equality holds
\[
\dim_{\mathbb{C}}
\mathcal{O}_{\mathbb{C}^{n},0}
/
\langle 
F_{1},\dots,F_{n-1},
f
\rangle
=
\sum_{k=1}^{\MM}m_{k}\ \ord_{0}f\circ n_{k}.
\]
\end{Theorem}
\begin{proof}
See \cite[p~78, Theorem~3]{DAngelo-Book-1993} for further discussion.
\end{proof}

\subsubsection{} We begin discussion with a small lemma.

\begin{Lemma}
Let $Z_{k}$ be an irreducible 1-dimensional analytic variety and $n_{k}:(\mathbb{C},0)\rightarrow (Z_{k},0)$ be its normalisation. Let $f$ be a holomorphic function germ vanishing at the origin such that $\ord_{0}f\circ n_{k}$ is finite. Then the intersection $Z_{k}\cap \{f=0\}$ is discrete, and hence the origin is an isolated point in the intersection.
\end{Lemma}
\begin{proof}

There is an equality of sets:
\[
\{f=0\}\cap Z_{k}=
\{n_{k}(\zeta):\ f\circ n_{k}(\zeta)=0\}.
\]
Now the set on the right is just simply $\{0\}$. This is because by hypothesis on the vanishing order of $f\circ n_{k}$,
\[
f\circ n_{k}(\zeta)=
\zeta^{m}g(\zeta),
\]
where $g(0)\neq 0$ and $m=\ord_{0}\ f\circ n_{k}<\infty$. Hence
\[
f\circ n_{k}(\zeta)=0 \implies \zeta=0\implies n_{k}(\zeta)=0.\qedhere
\]
\end{proof}

\begin{Proposition}
\label{multofideal-prop: finite order vanishing implies finite intersection multiplicity}
Let $f$ be the holomorphic function such that for each $1\leqslant k\leqslant \MM$, the vanishing order $\ord_{0}f\circ n_{k}$ is finite. Then the intersection multiplicity of the ideal $\langle F_{1},\dots,F_{n-1},f\rangle$ is finite.
\end{Proposition}

\begin{proof}
The previous lemma implies that  
\[
Z_{k}\cap \{f=0\}=\{0\}.
\]
Hence
\[
V(F_{1},\dots,F_{n-1},f)
=
C\cap \{f=0\}
=
\left(\bigcup_{k=1}^{\MM}Z_{k}\right)
\cap \{f=0\}
=
\bigcup_{k=1}^{\MM}(Z_{k}\cap \{f=0\})
=
\{0\}.\qedhere
\]
\end{proof}

\subsubsection{} We are now in a position to prove the following lemma.

\begin{Proposition}\label{multofideal-prop: multiplicity of an ideal for curves-explicit version}
Let $F_{1}$,\dots,$F_{n-1}$ be a regular sequence such that $V(F_{1},\dots,F_{n-1})$ is a pure $1$-dimensional variety. For a generic choice of hyperplane defined by a linear function $L$, 
\[
\dim_{\mathbb{C}}
\mathcal{O}_{\mathbb{C}^{n},0}
/
\langle F_{1},\dots, F_{n-1},L\rangle
=
\sum_{k=1}^{\MM}m_{k}\ 
\min\{\mu_{k1},\dots,\mu_{kn}\}.
\]
\end{Proposition}
\begin{proof}
First of all, $L$ may be written as
\[
L=\sum_{j=1}^{n}c_{j}z_{j}
\eqno
{\scriptstyle (c_{k}\,\in\,\mathbb{C})}.
\]
Suppose that the intersection multiplicity of the ideal $\langle F_{1},\dots, F_{n-1},L\rangle$ is finite, by Theorem \ref{multofideal-thm: compute intersection numbers by normalisation}, 
\[
\dim_{\mathbb{C}}
\mathcal{O}_{\mathbb{C}^{n},0}
/
\langle F_{1},\dots, F_{n-1},L\rangle
=
\sum_{k=1}^{\MM}m_{k}\ 
\ord_{0}L\circ n_{k}.
\]
By Proposition \ref{multofideal-prop: finite order vanishing implies finite intersection multiplicity}, it suffices to choose an appropriate $L$ such that $\ord\ L\circ n_{k}<\infty$. First, observe that
\begin{equation}\label{multofideal-eqn-1}
\begin{aligned}
L\circ n_{k}
 & = 
\sum_{j=1}^{n}c_{j}\zeta^{\mu_{kj}}
a_{kj}(\zeta)\\
 & = 
\zeta
^{\minsmall\{\mu_{k1},\dots,\mu_{kn}\}}
\sum_{\{j:\mu_{kj}=\minsmall\{\mu_{k1},\dots,\mu_{kn}\}\}}
c_{j}a_{kj}(\zeta)
+
O(\zeta
^{\minsmall\{\mu_{k1},\dots,\mu_{kn}\}+1}).
\end{aligned}
\end{equation}
If
\[
(c_{1},\dots, c_{n})
\in 
\mathbb{C}^{n}
-
\bigcup_{k=1}^{\MM}
\left\{(d_{1},\dots,d_{n})\in\mathbb{C}^{n}:\ 
\sum_{j=1}^{n}d_{j}a_{kj}(0)=0
\right\},
\]
then by equation \eqref{multofideal-eqn-1},
\[
\ord_{0}L\circ n_{k}=
\min\{\mu_{k1},\dots,\mu_{kn}\}<\infty.
\]
This completes the proof.
\end{proof}

\section{Generic Selection of Linear Combinations for Effective Termination}
The following proposition appears in \cite[p~1190]{Siu-2010}.
\begin{Proposition}\label{genselect-prop-main result}
Let $0\leqslant q\leqslant n$, and $f_{1},\dots,f_{q}$ be holomorphic function germs on $\mathbb{C}^{n}$ at the origin such that the common zero set $\{f_{1}=\dots=f_{q}=0\}$ is a pure $(n-q)$-dimensional variety germ in $\mathbb{C}^{n}$ at the origin. Let $m$ be the multiplicity of the ideal $\langle f_{1},\dots, f_{q}\rangle$ in the sense that for any $(n-q)$ generic homogeneous linear functions $L_{1}$,\dots,$L_{n-q}$,
\[
\dim_{\mathbb{C}}\ 
\mathcal{O}_{\mathbb{C}^{n},0}
/
\langle 
f_{1},\dots, f_{q},L_{1},\dots,
L_{n-q}\rangle
=
m.
\]
Let $V(f_{1},\dots,f_{q},L_{1},\dots, L_{n-q})$ be a pure $1$ dimensional analytic variety and let 
\[
V(f_{1},\dots,f_{q},L_{1},\dots, L_{n-q})
=
\bigcup_{k=1}^{\MM}Z_{k}
\]
be the irreducible decomposition of $V(f_{1},\dots,f_{q},L_{1},\dots, L_{n-q})$. Let $F_{1},\dots,F_{\NN}$ be holomorphic function germs in $\mathcal{O}_{\mathbb{C}^{n},0}$ vanishing at the origin and $p\geqslant 1$ be an integer such that 
\[
|z|^{p}\lesssim
\sum_{i=1}^{\NN}|F_{i}|.
\]
Then there exist $\MM$ hyperplanes, $H_{1}$,\dots,$H_{\MM}$, in $\mathbb{C}^{\NN}$ such that for any
\[
(c_{1},\dots,c_{\NN})
\in 
\mathbb{C}^{\NN}
-
\bigcup_{i=1}^{\MM}H_{i},
\]
and for any generic $(n-q-1)$ homogeneous linear functions $L_{1}$,\dots,$L_{n-q-1}$ the following inequality holds
\[
\dim_{\mathbb{C}}
\mathcal{O}_{\mathbb{C}^{n},0}
\bigg/
\left\langle
f_{1},\dots,f_{q},
\sum_{j=1}^{\NN}c_{j}F_{j},
L_{1},\dots,L_{n-q-1}
\right\rangle
\leqslant mp.
\]
\end{Proposition}

\begin{proof}
As in the statement of the proof, let
\[
V(f_{1},\dots,f_{q},L_{1},\dots, L_{n-q})
=
\bigcup_{k=1}^{\MM}Z_{k}
\]
be the irreducible decomposition of the pure $1$-dimensional analytic variety, and let 
\[
n_{k}:(\mathbb{C},0)
\longrightarrow
(Z_{k},0)
\]
be normalisations of $Z_{k}$. By Theorem \ref{multofideal-thm: compute intersection numbers by normalisation} and Proposition \ref{multofideal-prop: finite order vanishing implies finite intersection multiplicity}, there exist strictly positive integers $m_{1}$,\dots,$m_{\MM}$ such that for any holomorphic function germs $g\in\mathcal{O}_{\mathbb{C}^{n},0}$ with $\ord_{0}\ g\circ n_{k}<\infty$ for all $k$,
\[
\dim_{\mathbb{C}}\ 
\mathcal{O}_{\mathbb{C}^{n},0}
\big/
\langle 
f_{1},\dots,f_{q},
g,
L_{1},\dots,L_{n-q-1}
\rangle
=
\sum_{k=1}^{\MM}
m_{k}\ \ord_{0}\ g\circ n_{k}.
\]
It suffices to find suitable constants $(c_{1},\dots,c_{\NN})$ such that the order of vanishing of the following function
\[
g\circ n_{k}:=
\sum_{j=1}^{\NN}\ 
c_{j}F_{j}\circ n_{k}
\]
is finite for all $k$. For each fixed $1\leqslant k\leqslant \MM$, the map $n_{k}$ may be explicitly written as
\[
n_{k}:
\zeta
\longmapsto 
(\zeta^{\mu_{k,1}}a_{k,1}(\zeta),
\dots,
\zeta^{\mu_{k,n}}a_{k,n}(\zeta)
),
\]
where for each $1\leqslant l\leqslant n$, $a_{k,l}(0)\neq 0$. Let
\[
s_{k}
:=
\min\ \{\mu_{k,1},\dots,\mu_{k,n}\}
\eqno
{\scriptstyle{(1\,\leqslant\, k\,\leqslant\, \MM).}}
\]
Pulling back the inequality
\[
|z|^{p}\lesssim
\sum_{j=1}^{\NN}|F_{j}|
\]
by the normalisations give
\reqnomode\usetagform{EngelLie}
\begin{align*}
 |\zeta|^{s_{k}p}
& \lesssim  
|(\zeta^{\mu_{k,1}}a_{k,1}(\zeta),\dots,
\zeta^{\mu_{k,n}}a_{k,n}(\zeta))|^{p}
\notag
\\
&\leqslant  
 A\sum_{j=1}^{\NN}
 |F_{j}\circ n_{k}(\zeta)|
 \tag{(1\,\leqslant\, k\,\leqslant\,\MM).}
 \end{align*}
Consequently, not all $F_{j}\circ n_{k}$ vanish at the same time. For any $F_{j}\circ n_{k}\not\equiv 0$, the one-variable holomorphic function may be expanded into power series
 \[
 F_{j}\circ n_{k}
 =
 \sum_{l=t_{j,k}}^{\infty}
 F_{j,k,l}\zeta^{l}
 \eqno
 {\scriptstyle{(1\,\leqslant\, j\,\leqslant\, \NN,\,\,1\,\leqslant\, k\,\leqslant\, \MM),}}
 \]
 where $t_{j,k}=\ord_{0}\ F_{j}\circ n_{k}$ and $F_{j,k,l}\in\mathbb{C}$. By convention, $t_{j,k}=\infty$ if $F_{j}\circ n_{k}\equiv 0$. For a fixed $1\leqslant k\leqslant \MM$, let
 \[
 t_{k}
:=
 \min\{
 t_{j,k}:\ 
 1\leqslant j\leqslant \NN,\ 
 F_{j}\circ n_{k}\not\equiv 0\}
 <\infty.
 \]
 Hence
 \[
 |\zeta|^{s_{k}p}
 \lesssim
 |\zeta|^{t_{k}}
 \eqno
 {\scriptstyle{(1\,\leqslant\, k\,\leqslant\, \MM),}}
 \]
which implies that $t_{k}\leqslant s_{k}p$, since $|\zeta|\ll 1$. For any $(c_{1},\dots,c_{\NN})\in\mathbb{C}^{\NN}$, 
\reqnomode\usetagform{EngelLie}
 \begin{align}
 \sum_{j=1}^{\NN}c_{j}F_{j}\circ n_{k}
 &= \sum_{j=1}^{\NN}
 \sum_{l\geqslant t_{j,k}}^{\infty}
 c_{j}F_{j,k,l}\zeta^{l}
 \notag
 \\
 &=
 \left(
 \sum_{\{j:\ t_{j,k}=t_{k}\}}
 c_{j}F_{j,k,t_{k}}
 \right)
 \zeta^{t_{k}}
+
O(\zeta^{t_{k}+1})
\tag{(1\,\leqslant\, k\,\leqslant\, \MM).}
 \end{align}
 Therefore, the order of vanishing of $\sum_{j=1}^{\NN}c_{j}F_{j}\circ n_{k}$ is exactly $t_{k}$ if $\sum_{\{j:t_{j,k}=t_{k}\}}
 c_{j}F_{j,k,t_{k}}\neq 0$.
 If
 \[
 (c_{1},\dots,c_{\NN})\in\mathbb{C}^{\NN}
 -
 \bigcup_{k=1}^{\MM}
 \left\{
 (d_{1},\dots,d_{\NN})\in\mathbb{C}^{\NN}:\ 
 \sum_{\{j:\ t_{j,k}=t_{k}\}}
 d_{j}F_{k,j,t_{k}}
 =
 0
 \right\}
 \]
 which in the complement of the union of $\MM$ hyperplanes, then for each $1\leqslant k\leqslant \MM$,
 \[
 \ord_{0}F\circ n_{k}
 =
 t_{k}
 \leqslant s_{k}p.
 \]
 Consequently, 
  \begin{eqnarray*}
 &&\dim_{\mathbb{C}}
 \mathcal{O}_{\mathbb{C}^{n},0}
 \bigg/
 \left\langle f_{1},\dots,f_{q},
 \tilde{F}_{1},\dots,\tilde{F}_{\nu},\sum_{j=1}^{\NN}c_{j}F_{j},
 L_{1},\dots,L_{n-q-\nu-1}
 \right \rangle\\
 &=&
 \sum_{k=1}^{\MM}
 m_{k}\ \ord_{0}
 \sum_{j=1}^{\NN}
 c_{j}F_{j}\circ n_{k}\\
 &=&
 \sum_{k=1}^{\MM}m_{k}t_{k}
\leqslant 
 \sum_{k=1}^{\MM}
 m_{k}s_{k}p\\
 &&
 \ \ \ \ \ \ \ \ \ \ \ \ \ \ \ \ =p\sum_{k=1}^{\MM}m_{k}s_{k}\\
 &&
 \ \ \ \ \ \ \ \ \ \ \ \ \ \ \ \ 
 =
 p\sum_{k=1}^{\MM}m_{k}
 \min\{\mu_{k,1},\dots,\mu_{k,n}\}.
 \end{eqnarray*}
 By Proposition \ref{multofideal-prop: multiplicity of an ideal for curves-explicit version}, the number $\sum_{k=1}^{\MM}m_{k}
 \min\{\mu_{k,1},\dots,\mu_{k,n}\}$ is the intersection multiplicity of the curve $V(f_{1},\dots,f_{q},L_{1},\dots,L_{n-q-1})$ with a generic hyperplane defined by $\{L_{n-q}=0\}$. By hypothesis,
 \[
 \sum_{k=1}^{\MM}m_{k}
 \min\{\mu_{k,1},\dots,\mu_{k,n}\}=m
 \]
 and this completes the proof.
\end{proof}

\subsubsection{In dimension $2$} We will state the corollary of Proposition \ref{genselect-prop-main result} in the case of dimension $2$.

\begin{Corollary}\label{genselect-cor-main result}
Let $F_{1}$, \dots, $F_{\NN}$ be holomorphic functions in $\mathcal{O}_{\mathbb{C}^{2},0}$ such that the ideal $\mathcal{I}_{F}=\langle F_{1},\dots,F_{\NN}\rangle$ has finite intersection multiplicity with data $(p,q,s)$.
Then there exist generic constants $(c_{1},\dots, c_{\NN})\in\mathbb{C}^{\NN}$ such that 
\[
\mult_{0}\left(
\sum_{j=1}^{\NN}c_{j}F_{j}
\right)
\leqslant q
\leqslant 4p.
\]
Moreover, let $V(\tilde{F}_{1})=\cup_{k=1}^{\MM}Z_{k}$ be the irreducible decomposition of the variety. Then there exist $\MM$ hyperplanes $H_{1}$, \dots,$H_{\MM}$ in $\mathbb{C}^{\NN}$ such that for all $(d_{1},\dots,d_{\NN})\in\mathbb{C}^{\NN}-
\cup_{i=1}^{\MM}H_{i}$,
\[
\dim_{\mathbb{C}}\ 
\mathcal{O}_{\mathbb{C}^{n},0}
\bigg/
\left\langle
\sum_{j=1}^{\NN}c_{j}F_{j},\ 
\sum_{j=1}^{\NN}d_{j}F_{j}
\right\rangle
\leqslant 4p^{2}\leqslant 4s^{2}.
\]
\end{Corollary}
\begin{proof}
First, there exists $1\leqslant i\leqslant \NN$ such that $\mult_{0}\ F_{i}\leqslant q$. Otherwise, if $\mult_{0}\ F_{i}\geqslant q+1$ for every $1\leqslant i\leqslant \NN$, then 
\[
\mathfrak{m}^{q}
\subseteq
\langle F_{1},\dots,F_{\NN}\rangle
\subseteq
\mathfrak{m}^{q+1},
\]
which is a contradiction. So let $(c_{1},\dots,c_{\NN})$ be constants so that 
\[
\mult_{0}\left(
\sum_{j=1}^{\NN}c_{j}F_{j}
\right)
\leqslant q\leqslant 4p,
\]
where the last inequality follows from Proposition \ref{LAG-prop-Relations between the intersection invariants}. Then the existence of $\MM$ hyperplanes in $\mathbb{C}^{\NN}$ and constants $(d_{1},\dots, d_{\NN})$ so that the conclusion holds follow directly from the previous propostion, and the proof is complete.
\end{proof}
\section{Proper Maps and Projections}
\subsubsection{}\label{propmap-para-1} In this section, let $h_{1}$,\dots,$h_{n}$ be holomorphic function germs in $\mathcal{O}_{\mathbb{C}^{n},0}$ vanishing at the origin with 
\[
\dim_{\mathbb{C}}
\mathcal{O}_{\mathbb{C}^{n},0}
\big/
\langle h_{1},\dots, h_{n}\rangle
=:
s
<\infty.
\]
Hence the $(n-1)$-tuple $(h_{1},\dots, h_{n-1})$ forms a regular sequence. By Proposition \ref{multofideal-prop: multiplicity of an ideal for curves-explicit version}, and by a suitable linear change of coordinates, there exists a positive integer $m$ such that 
\[
\dim_{\mathbb{C}}
\mathcal{O}_{\mathbb{C}^{n},0}
\big/
\langle h_{1},\dots, h_{n-1},z_{n}\rangle
=:m
\]
which is the multiplicity of the ideal $\langle h_{1},\dots,h_{n-1}\rangle$.

\subsubsection{}\label{propmap-para-2} The map
\begin{eqnarray*}
\varphi:(\mathbb{C}^{n},0) 
&\longrightarrow & 
(\mathbb{C}^{n},0)\\
(z_{1},\dots,z_{n})
&\longmapsto & 
(h_{1}(z),\dots,h_{n-1}(z),z_{n})=:
(w_{1},\dots,w_{n-1},w_{n})
\end{eqnarray*}
is proper and open with finite fibres. Let $H=\{h_{n}=0\}$ be the hypersurface defined by the zeros of $h_{n}$. By Remmert's proper mapping theorem\footnote{Remmert's proper mapping theorem may be stated as follows: if $M$ and $N$ are complex manifolds, $f:M\rightarrow N$ a holomorphic map and $V\subset M$ an analytic variety such that $f|_{V}$ is proper, then $f(V)$ is an analytic subvariety of $N$.}, the image $\varphi(H)$ is also an analytic set. Since the map restricted to the hypersurface $H$:
\[
\varphi|_{H}:H\longrightarrow 
\varphi(H)
\]
is surjective with finite fibres, by  section 4, paragraph \ref{idealgrad-para-8} (or \cite[p~129, Lemma~4.1.4]{deJong-Pfister-2000}), one has $\dim\ H=\dim\ \varphi(H)=n-1$.
\subsubsection{}\label{propmap-para-3} Since $\varphi(H)$ is of dimension $n-1$, it is a  hypersurface locally defined at the origin by a certain holomorphic function $\tilde{h}_{n}(w_{1},\dots,w_{n})$, which will be shown to have the following properties:

\smallskip\noindent{\bf (i)}
$\tilde{h}_{n}(0,\dots, 0,w_{n})\not\equiv 0$ with certain order of vanishing $\lambda:=\ord_{0}(\tilde{h}_{n}(0,z_{n}))$. By the Weierstrass Preparation Theorem, $\tilde{h}_{n}$ may be expressed as
\[
\tilde{h}_{n}(w)
=
u(w)
\left(
w_{n}^{\lambda}
+
\sum_{j=0}^{\lambda-1}
a_{j}(w_{1},\dots,w_{n-1})w_{n}^{j}
\right),
\]
for some unit $u(w)$, and $a_{j}(0)=0$ for all $0\leqslant j\leqslant \lambda-1$. 

\smallskip\noindent{\bf (ii)}
$\lambda\leqslant s$.
\subsubsection{}\label{propmap-para-4}
\begin{Lemma}\label{propmap-lem-projection from H to first n-1 coordinates has finite distinct fibres}
Let $h_{1}$,\dots,$h_{n}$ be holomorphic function germs in $\mathcal{O}_{\mathbb{C}^{n},0}$ vanishing at the origin such that the intersection multiplicity of the ideal $\langle h_{1},\dots,h_{n}\rangle$ is finite with data $(p,q,s)$. Let $H:=\{h_{n}=0\}$ be the hypersurface defined as the vanishing locus of $h_{n}$. Consider the map:
\begin{eqnarray*}
\psi: H &\longrightarrow & \mathbb{C}^{n-1}\\
z:=(z_{1},\dots,z_{n}) &\longmapsto & 
(h_{1}(z),\dots,h_{n-1}(z)).
\end{eqnarray*}
Then there exists a open neighbourhood $U\subseteq\mathbb{C}^{n-1}$ of the origin $0\in\mathbb{C}^{n-1}$ such that for every $\alpha:=(\alpha_{1},\dots,\alpha_{n-1})\in U$, there are at most $s$ distinct elements in $\psi^{-1}(\alpha)$.
\end{Lemma}
\begin{proof}
We prove by contradiction. Suppose for every open neighbourhood $U\subseteq \mathbb{C}^{n-1}$ of the origin $0\in\mathbb{C}^{n-1}$, there exists a point $\alpha\in U$ such that the number of distinct elements in $\psi^{-1}(\alpha)$ is at least $s+1$. 

By hypothesis, the map
\begin{eqnarray*}
\Psi:\mathbb{C}^{n}
&\longrightarrow &
\mathbb{C}^{n}\\
z 
&\longmapsto & 
(h_{1}(z),\dots,h_{n}(z))
\end{eqnarray*}
is a ramified $s$-sheeted analytic covering map. Hence, there exists a neighbourhood $V=V'\times V''\subseteq \mathbb{C}^{n-1}\times \mathbb{C}$ of the origin $0\in\mathbb{C}^{n}$ such that for every $\beta:=(\beta_{1},\dots, \beta_{n})\in V$, the number of distinct points in $\Psi^{-1}(\beta)$ is at most $s$.

But by our assumption, given $V'$ a neighbourhood of the origin $0\in\mathbb{C}^{n-1}$, there exists a point $\alpha:=(\alpha_{1},\dots,\alpha_{n-1})\in V'$ such that there are at least $s+1$ distinct points in $\psi^{-1}(\alpha)$. Since $(\alpha,0)\in V$ and 
\[
\Psi^{-1}(\alpha,0)=\psi^{-1}(\alpha),
\]
there are at least $s+1$ distinct points in $\Psi^{-1}(\alpha,0)$, which is a contradiction.
\end{proof}

\subsubsection{}\label{propmap-para-5} We will therefore answer the first claim in paragraph $3$.

\begin{Proposition}
Let $h_{1}$,\dots,$h_{n}$ be holomorphic function germs in $\mathcal{O}_{\mathbb{C}^{n},0}$ vanishing at the origin such that the multiplicity of the ideal $\langle h_{1},\dots,h_{n}\rangle$ is $s\in\mathbb{N}_{\geqslant 1}$. Suppose that the holomorphic map
\begin{eqnarray*}
\varphi:\mathbb{C}^{n} &\longrightarrow & 
\mathbb{C}^{n}\\
(z_{1},\dots,z_{n}) &\longmapsto &
(h_{1}(z),\dots,h_{n-1}(z),z_{n})
\end{eqnarray*}
defines a ramified $k$-sheeted covering for some positive integer $k$. Let $\tilde{h}_{n}$ be a  holomorphic function such that $\varphi(\{h_{n}=0\})=\{\tilde{h}_{n}=0\}$. Then $\tilde{h}_{n}(0,z_{n})\not\equiv 0$.
\end{Proposition}

\begin{proof}
Suppose on the contrary that $\tilde{h}_{n}(0,z_{n})\equiv 0$. Consider the composition of maps
\begin{eqnarray*}
H\stackrel{\varphi}{\xrightarrow{\hspace*{1cm}}}  & \varphi(H) &\mathop{\xrightarrow{\hspace*{1.5cm}}}_{}^
{\projsmall} \mathbb{C}^{n-1}\\
(z_{1},\dots,z_{n})  \xmapsto{\hspace*{1cm}} & (h_{1}(z),\dots,h_{n-1}(z),z_{n})\\
& (w_{1},\dots,w_{n}) &\xmapsto{\hspace*{1.5cm}} (w_{1},\dots,w_{n-1}).
\end{eqnarray*}
Here $\varphi$ is the map in the statement of the proposition and $\projsmall$ is the projection onto the first $n-1$ coordinates. Above $0\in\mathbb{C}^{n-1}$, since $\tilde{h}_{n}(0,z_{n})\equiv 0$,
\[
\{(0,z_{n})\in\mathbb{C}^{n}:\ z_{n}\in\mathbb{C}\}
\subseteq
\{\tilde{h}_{n}=0\}
=
\varphi(H).
\]
Moreover, since $\projsmall(0,z_{n})=0\in\mathbb{C}^{n-1}$, 
\[
\{(0,z_{n})\in\mathbb{C}^{n}:\ z_{n}\in\mathbb{C}\}
\subseteq
\projsmall^{-1}(0).
\]
Therefore, $\projsmall^{-1}(0)$ has infinitely many distinct fibre points. Consequently, $(\projsmall\circ\varphi)^{-1}(0)$ has infinitely many distinct fibre points. But $\proj\circ\varphi=\psi$ in the previous lemma, has finite distinct fibres, contradiction.
\end{proof}

\subsubsection{}\label{propmap-para-6} Next, we will show that $\ord_{0}\ \tilde{h}_{n}(0,z_{n})\leqslant s$.

\begin{Lemma}\label{propmap-lem-finiteness by Weierstrass prep theorem}
Let $\tilde{h}_{n}$ be a holomorphic function germ in $\mathcal{O}_{\mathbb{C}^{n},0}$ with $\tilde{h}_{n}(0)=0$, and $\tilde{h}_{n}(0,z_{n})\not\equiv 0$ so that $\ord_{0}\ \tilde{h}_{n}(0,z_{n})<\infty$. If the projection 
\begin{eqnarray*}
\pi: \{\tilde{h}_{n}=0\} 
&\longrightarrow & (\mathbb{C}^{n-1},0)\\
(\alpha_{1},\dots, \alpha_{n})
&\longmapsto & (\alpha_{1},\dots,\alpha_{n-1})
\end{eqnarray*}
is a finite surjective map with at most $s$ distinct fibre points above each point in $(\mathbb{C}^{n-1},0)$, then $\ord_{0}\ \tilde{h}_{n}(0,z_{n})\leqslant s$.
\end{Lemma}

\begin{proof}
Suppose on the contrary that $\lambda:=\ord_{0}\ \tilde{h}_{n}(0,z_{n})\geqslant s+1$. By the hypothesis that $\lambda<\infty$, Weierstrass Preparation Theorem implies the existence of a unit $u(z_{1},\dots,z_{n})$ and $\ord_{0}\ \tilde{h}_{n}(0,z_{n})$ holomorphic functions $a_{j}(z_{1},\dots,z_{n-1})$ vanishing at $(z_{1},\dots,z_{n-1})=(0,\dots,0)$ such that 
\[
h(z_{1},\dots,z_{n})
=
u(z_{1},\dots,z_{n})
\left(
z_{n}^{\lambda}
+
\sum_{j=0}^{\lambda-1}
a_{j}(z_{1},\dots,z_{n-1})\ z_{n}^{j}
\right).
\]

Therefore, above a generic point $(\alpha_{1},\dots,\alpha_{n-1})\in\mathbb{C}^{n-1}$, the preimages $(\alpha_{1},\dots,\alpha_{n-1},z_{n})$ of $\pi$ which must satisfy the following polynomial equation
\[
z_{n}^{\lambda}
+
\sum_{j=0}^{\lambda-1}
a_{j}(\alpha_{1},\dots,\alpha_{n-1})\ z_{n}^{j}=0
\]
has $\lambda\geqslant s+1$ distinct solutions in $z_{n}$. This contradicts the hypothesis in the statement of the lemma.
\end{proof}


\begin{Proposition}\label{propmap-prop-main result}
Let $h_{1}$,\dots,$h_{n}$ be holomorphic function germs in $\mathcal{O}_{\mathbb{C}^{n},0}$ vanishing at the origin such that 
\[
\dim_{\mathbb{C}}\ 
\mathcal{O}_{\mathbb{C}^{n},0}
\big/
\langle h_{1},\dots,h_{n}\rangle
=s<\infty.
\]
Let $H=\{h_{n}=0\}$. Suppose that the holomorphic map 
\begin{eqnarray*}
\varphi:\mathbb{C}^{n} &\longrightarrow & 
\mathbb{C}^{n}\\
(z_{1},\dots,z_{n}) 
&\longmapsto & 
(h_{1}(z),\dots,h_{n-1}(z),z_{n})
\end{eqnarray*}
is  proper, open so that there exists a holomorphic function $\tilde{h}_{n}(w_{1},\dots,w_{n})$ with $\varphi(H)=\{\tilde{h}_{n}=0\}$. Then $\ord_{0}\ \tilde{h}_{n}(0,\dots,0,w_{n})\leqslant s$.
\end{Proposition}

\begin{proof}
Consider the map
\begin{eqnarray*}
H \stackrel{\varphi|_{H}}{\xrightarrow{\hspace*{1.5cm}}}  & \varphi(H) & \stackrel{\projsmall}{\xrightarrow{\hspace*{1.5cm}}}\mathbb{C}^{n-1}\\
(z_{1},\dots,z_{n})  \xmapsto{\hspace*{1.5cm}} & (h_{1}(z),\dots,h_{n-1}(z),z_{n})\\
& (w_{1},\dots,w_{n}) &\xmapsto{\hspace*{1.5cm}} (w_{1},\dots,w_{n-1}).
\end{eqnarray*}
By lemma \ref{propmap-lem-projection from H to first n-1 coordinates has finite distinct fibres}, there exists a neighbourhood $U$ of the origin $0\in\mathbb{C}^{n-1}$ such that for all $\alpha\in U$, there are at most $s$ distinct points in $\psi^{-1}(\alpha)=(\proj\circ\varphi|_{H})^{-1}(\alpha)$. Choose a generic point $\alpha\in\mathbb{C}^{n-1}$ as in the lemma \ref{propmap-lem-finiteness by Weierstrass prep theorem}. Therefore, above $\alpha$, there are $\ord_{0}\ \tilde{h}_{n}(0,w_{n})$ distinct fibre points in $\proj^{-1}(\alpha)$, and hence
\begin{eqnarray*}
&& \ord_{0}\ \tilde{h}_{n}(0,w_{n})\\
&=&
\text{number of distinct points in }\proj^{-1}(\alpha)\\
&\leqslant & 
\text{number of distinct points in }(\proj\circ\varphi|_{H})^{-1}(\alpha)\leqslant s.\qedhere
\end{eqnarray*}
\end{proof}

\section{Calculation of Explicit $\varepsilon$ in Dimension $2$ (Preliminaries)}
\subsubsection{}\label{calprelim-para-1} In this section we will use some of the results in the earlier sections to establish some preliminary results for the calculation of explicit $\varepsilon$ in the case of dimension $2$.

\subsubsection{}\label{calprelim-para-2} Let $F_{1}$,\dots,$F_{\NN}$ be holomorphic function germs in $\mathcal{O}_{\mathbb{C}^{2},0}$ vanishing at the origin such that the ideal they generate $\langle F_{1},\dots, F_{\NN}\rangle$ has finite intersection multiplicity with data $(p,q,s)$.

\subsection{Ideal Generated by Gradient and Generic Selection in Dimension $2$}

\subsubsection{}\label{calprelim-para-1-1} In $\mathbb{C}^{2}$, Proposition \ref{idealgrad-prop-ideal of gradient has finite intersection multiplicity} implies that
\[
\dim_{\mathbb{C}}\ 
\mathcal{O}_{\mathbb{C}^{2},0}
\bigg/
\left\langle
\frac{\partial F_{i}}{\partial z_{j}}:\ 1\leqslant i\leqslant \NN,\ 1\leqslant j\leqslant 2
\right\rangle
\leqslant
\binom{
8s+1}{
8s-1}.
\]
Moreover, if $\NN=2$, there is a better upper bound
\[
\dim_{\mathbb{C}}\ 
\mathcal{O}_{\mathbb{C}^{2},0}
\bigg/
\left\langle
\frac{\partial F_{i}}{\partial z_{j}}:\ 1\leqslant i,j\leqslant 2
\right\rangle
\leqslant
4s.
\]
\subsubsection{}\label{calprelim-para-1-2} Let $\tilde{h}_{2}$ be any holomorphic function germ in $\mathcal{O}_{\mathbb{C}^{2},0}$ with multiplicity $\tilde{m}_{2}$. Let 
\[
V(\tilde{h}_{2})
=
\bigcup_{k=1}^{\tilde{r}}Z_{k}
\]
be the irreducible decomposition of the pure $1$-dimensional analytic variety. By Proposition \ref{genselect-prop-main result}, there exist $\tilde{r}$ hyperplanes $H_{1}$,\dots,$H_{\tilde{r}}$ in $\mathbb{C}^{2\NN}$ such that for all 
\[
(\lambda_{1},\dots,\lambda_{\NN},
\theta_{1},\dots,\theta_{\NN})
\in 
\mathbb{C}^{2\NN}
-
\bigcup_{i=1}^{\tilde{r}}H_{i},
\]
there is an effective upper bound on the intersection multiplicity
\[
\dim_{\mathbb{C}}\ 
\mathcal{O}_{\mathbb{C}^{2},0}
\bigg/
\left\langle
\tilde{h}_{2},\ 
\sum_{j=1}^{\NN}
\lambda_{j}\frac{\partial F_{j}}{\partial z_{1}}
+
\theta_{j}\frac{\partial F_{j}}{\partial z_{2}}
\right\rangle
\leqslant
\tilde{m}_{2}
\binom{
8s+1}{
8s-1}.
\]

\subsubsection{}
\begin{Lemma}\label{calprelim-lem-generic selection for effectiveness}
Let $\tilde{h}_{2}$ be any holomorphic function germ in $\mathcal{O}_{\mathbb{C}^{2},0}$ that vanishes at the origin, whose multiplicity is $\tilde{m}_{2}$. Suppose that the vanishing locus $\{\tilde{h}_{2}=0\}$ is a union of $\tilde{r}$ irreducible components (not counting multiplicity). Then there exist $2\tilde{r}$ hyperplanes $H_{1}$,\dots,$H_{2\tilde{r}}$ in $\mathbb{C}^{\NN}$ so that whenever
\[
(c_{1},\dots,c_{\NN})
\in 
\mathbb{C}^{\NN}
-
\bigcup_{k=1}^{2\tilde{r}}H_{k}
\]
there are $\tilde{r}$ hyperplanes $\tilde{H}_{1}$,\dots,$\tilde{H}_{\tilde{r}}$ in $\mathbb{C}^{2}$ such that if 
\[
(\alpha,\gamma)
\in 
\mathbb{C}^{2}
-
\bigcup_{k=1}^{\tilde{r}}
\tilde{H}_{k},
\]
then it holds that 
\[
\dim_{\mathbb{C}}\ 
\mathcal{O}_{\mathbb{C}^{2},0}
\bigg/
\left\langle
\tilde{h}_{2},\ 
\sum_{j=1}^{\NN}
c_{j}\alpha
\frac{\partial F_{j}}{\partial z_{1}}
+
c_{j}\gamma
\frac{\partial F_{j}}{\partial z_{2}}
\right\rangle
\leqslant
\tilde{m}_{2}
\left(
\begin{matrix}
8s+1\\
8s-1
\end{matrix}
\right).
\]
\end{Lemma}
\begin{proof}
By paragraph \ref{calprelim-para-1-1}, the ideal 
\[
\left\langle 
\frac{\partial F_{i}}{\partial z_{j}}:\ 
1\leqslant i\leqslant\NN,\ 
1\leqslant j\leqslant 2
\right\rangle
\]
has finite intersection multiplicity with data $(p',q',s')$.

By Proposition \ref{genselect-prop-main result}, there exist $\tilde{r}$ hyperplanes in $\mathbb{C}^{2\NN}$ of the form
\[
H_{l}'
=
\left\{
(v_{1},\dots,v_{\NN},w_{1},\dots,w_{\NN})\in\mathbb{C}^{2\NN}:\ 
\sum_{k=1}^{\NN}
\sigma_{lk}v_{k}
+
\mu_{lk}w_{k}=0
\right\}
\eqno
{\scriptstyle{(1\,\leqslant\, l\,\leqslant\, \tilde{r}),}}
\]
such that if $(v_{1},\dots,v_{\NN},w_{1},\dots,w_{\NN})\in\mathbb{C}^{2\NN}
-\cup_{l=1}^{\tilde{r}}H_{l}'$, then 
\[
\dim_{\mathbb{C}}\ 
\mathcal{O}_{\mathbb{C}^{2},0}
\bigg/
\left\langle
\tilde{h}_{2},\ 
\sum_{k=1}^{2\NN}
v_{k}\frac{\partial F_{k}}{\partial z_{1}}
+
w_{k}\frac{\partial F_{k}}{\partial z_{2}}
\right\rangle
\leqslant 
\tilde{m}_{2}
\left(
\begin{matrix}
8s+1\\
8s-1
\end{matrix}
\right).
\]
To conclude the proof, it suffices to choose $(c_{1}\alpha,\dots,c_{\NN}\alpha,c_{1}\gamma,\dots,c_{\NN}\gamma)\in 
\mathbb{C}^{2\NN}
-
\cup_{l=1}^{\tilde{r}}H_{l}'$, or equivalently for every $1\leqslant l\leqslant \tilde{r}$,
\begin{eqnarray}\label{calprelim-eqn-1}
\sum_{k=1}^{\NN}
\sigma_{lk}c_{k}\alpha
+
\mu_{lk}c_{k}\gamma
\neq 0.
\end{eqnarray}
To this aim, write
\begin{eqnarray}\label{calprelim-eqn-2}
\sum_{k=1}^{\NN}
\sigma_{lk}c_{k}\alpha
+
\mu_{lk}c_{k}\gamma
=
\left(
\sum_{k=1}^{\NN}
\sigma_{lk}c_{k}
\right)\alpha
+
\left(
\sum_{k=1}^{\NN}
\mu_{lk}c_{k}
\right)\gamma.
\end{eqnarray}
If  
\begin{eqnarray*}
(c_{1},\dots,c_{\NN})
\in 
\mathbb{C}^{\NN}
&-&
\bigcup_{l=1}^{\tilde{r}}
\left\{
(d_{1},\dots,d_{\NN})\in\mathbb{C}^{\NN}:\ 
\sum_{k=1}^{\NN}
\sigma_{lk}d_{k}=0
\right\}\\
&-&
\bigcup_{l=1}^{\tilde{r}}
\left\{
(d_{1},\dots,d_{\NN})\in\mathbb{C}^{\NN}:\ 
\sum_{k=1}^{\NN}
\mu_{lk}d_{k}=0
\right\},
\end{eqnarray*}
which is in a complement of $2\tilde{r}$ hyperplanes, the coefficients of $\alpha$ and $\gamma$ in the equation \eqref{calprelim-eqn-2} do not vanish. Once $(c_{1},\dots,c_{\NN})$ is chosen, if
\[
(\alpha,\gamma)
\in 
\mathbb{C}^{2}
-
\bigcup_{l=1}^{\tilde{r}}
\left\{
\left(
\sum_{k=1}^{\NN}
\sigma_{lk}c_{k}
\right)\alpha
+
\left(
\sum_{k=1}^{\NN}
\mu_{lk}c_{k}
\right)\gamma
=0
\right\},
\]
which lies in the complement of $\tilde{r}$ hyperplanes in $\mathbb{C}^{2}$, then equation \eqref{calprelim-eqn-1} holds. Hence the proof is complete.
\end{proof}

\subsubsection{}
\begin{Proposition}\label{calprelim-prop-prelim for calculations-main results}
Let $(z_{1},z_{2})\in\mathbb{C}^{2}$ be holomorphic coordinates in $\mathbb{C}^{2}$. Let $\tilde{h}_{2}$ be a holomorphic function germ in $\mathcal{O}_{\mathbb{C}^{2},0}$ vanishing at the origin with multiplicity $\tilde{m}_{2}$, and  suppose that its vanishing locus $\{\tilde{h}_{2}=0\}$ has $\tilde{r}$ irreducible components (not counting multiplicity).

Let $F_{1}$,\dots,$F_{\NN}$ be holomorphic function germs which generate an ideal $\langle F_{1},\dots, F_{\NN}\rangle$ having finite intersection multiplicity with data $(p,q,s)$. 

Let 
\begin{eqnarray*}
\left(
\begin{matrix}
z_{1}\\
z_{2}
\end{matrix}
\right)
=
\left(
\begin{matrix}
\alpha & \beta\\
\gamma & \delta
\end{matrix}
\right)
\left(
\begin{matrix}
w_{1}\\
w_{2}
\end{matrix}
\right)
\end{eqnarray*}
be an invertible linear change of coordinates. Then there are $3\tilde{r}$ hyperplanes $H_{1}$,\dots,$H_{3\tilde{r}}$ in $\mathbb{C}^{\NN}$ such that for each $(c_{1},\dots,c_{\NN})\in\mathbb{C}^{\NN}
-\cup_{k=1}^{3\tilde{r}}H_{k}$,
there exist $\tilde{r}$ hyperplanes $\tilde{H}_{1}$,\dots,$\tilde{H}_{\tilde{r}}$ and a hypersurface defined by a homogeneous polynomial $\{P=0\}$ such that whenever
\[
(\alpha,\gamma)
\in 
\mathbb{C}^{2}
-
\bigcup_{k=1}^{\tilde{r}}\tilde{H}_{k}
-
\{P=0\},
\]
 the linear combination 
\[
h_{1}(z_{1},z_{2})
=
\sum_{j=1}^{\NN}c_{j}F_{j}(z_{1},z_{2})
\]
will satisfy the following conditions:

\smallskip\noindent{\bf (i)}
the intersection multiplicity of the  ideal $\langle h_{1},\tilde{h}_{2}\rangle$ has an effective bound:
\[
\dim_{\mathbb{C}}\ 
\mathcal{O}_{\mathbb{C}^{2},0}
\big/
\langle h_{1},\tilde{h}_{2}\rangle
\leqslant\tilde{m}_{2}p
\leqslant \tilde{m}_{2}s;
\]
\smallskip\noindent{\bf (ii)}
in the new coordinates $(w_{1},w_{2})$,
\begin{eqnarray*}
&& \dim_{\mathbb{C}}
\mathcal{O}_{\mathbb{C}^{2},0}
\bigg/
\left\langle
\tilde{h}_{2}(\alpha w_{1}+\beta w_{2},\gamma w_{1}+\delta w_{2}),\ 
\frac{\partial h_{1}(\alpha w_{1}+\beta w_{2},\gamma w_{1}+\delta w_{2})}{\partial w_{1}}
\right\rangle\\
&\leqslant &
\tilde{m}_{2}
\binom{
8s+1}{
8s-1}; 
\end{eqnarray*}

\smallskip\noindent{\bf (iii)}
the holomorphic map induced from the change of coordinates
\begin{eqnarray*}
\varphi:
\mathbb{C}^{2} &\longrightarrow & 
\mathbb{C}^{2}\\
(w_{1},w_{2}) &\longmapsto & 
\big(h_{1}(\alpha w_{1}+\beta w_{2},
\gamma w_{1}+\delta w_{2}),w_{2}\big)
\end{eqnarray*}
is a covering map with finite fibres.
\end{Proposition}

\begin{proof}
{\bf (i)} By Proposition \ref{genselect-prop-main result}, there exist $\tilde{r}$ hyperplanes $H_{1}$,\dots,$H_{\tilde{r}}$ in $\mathbb{C}^{\NN}$ so that for all 
\[
(c_{1},\dots,c_{\NN})
\in 
\mathbb{C}^{\NN}
-
\bigcup_{k=1}^{\tilde{r}}
H_{k},
\]
one has (in variables $(z_{1},z_{2})$)
\[
\dim_{\mathbb{C}}\ 
\mathcal{O}_{\mathbb{C}^{2},0}
\bigg/
\left\langle
\tilde{h}_{2},\ 
\sum_{j=1}^{\NN}c_{j}F_{j}
\right\rangle
\leqslant 
\tilde{m}_{2}p
\leqslant 
\tilde{m}_{2}s.
\]
This satisfies the first condition, which remains unchanged even after a linear change of coordinates $(z_{1},z_{2})\leftrightarrow (w_{1},w_{2})$.

{\bf (ii)} After a change of variables,
\begin{eqnarray*}
\frac{\partial h_{1}(\alpha w_{1}+\beta w_{2},
\gamma w_{1}+\delta w_{2})}{\partial w_{1}}
&=&
\frac{\partial h_{1}(z_{1},z_{2})}{\partial z_{1}}\frac{\partial z_{1}}{\partial w_{1}}
+
\frac{\partial h_{1}(z_{1},z_{2})}{\partial z_{2}}\frac{\partial z_{2}}{\partial w_{1}}\\
&=&
\alpha \frac{\partial h_{1}(z_{1},z_{2})}{\partial z_{1}}
+
\gamma
\frac{\partial h_{1}(z_{1},z_{2})}{\partial z_{2}}\\
&=&
\sum_{j=1}^{\NN}
c_{j}\alpha \frac{\partial F_{j}(z_{1},z_{2})}{\partial z_{1}}
+
c_{j}\gamma \frac{\partial F_{j}(z_{1},z_{2})}{\partial z_{2}}.
\end{eqnarray*}
By Lemma \ref{calprelim-lem-generic selection for effectiveness}, there exist $2\tilde{r}$ hyperplanes $H_{\tilde{r}+1}$,\dots,$H_{3\tilde{r}}$ in $\mathbb{C}^{\NN}$ such that whenever
\[
(c_{1},\dots,c_{\NN})\in 
\mathbb{C}^{\NN}-
\bigcup_{k=\tilde{r}+1}^{3\tilde{r}}
H_{k},
\]
there are $\tilde{r}$ hyperplanes $\tilde{H}_{1}$,\dots, $\tilde{H}_{\tilde{r}}$ in $\mathbb{C}^{2}$ so that if 
\[
(\alpha,\gamma)
\in
\mathbb{C}^{2}-
\bigcup_{k=1}^{\tilde{r}}
\tilde{H}_{k},
\]
then
\[
\dim_{\mathbb{C}}\ 
\mathcal{O}_{\mathbb{C}^{2},0}
\bigg/
\left\langle
\tilde{h}_{2}(z_{1},z_{2}),\ 
\sum_{j=1}^{\NN}
c_{j}\alpha
\frac{\partial F_{j}(z_{1},z_{2})}{\partial z_{1}}
+
c_{j}\gamma
\frac{\partial F_{j}(z_{1},z_{2})}{\partial z_{2}}
\right\rangle
\leqslant
\tilde{m}_{2}
\binom{
8s+1}{
8s-1},
\]
or in other words,
\[
\dim_{\mathbb{C}}\ 
\mathcal{O}_{\mathbb{C}^{2},0}
\bigg/
\left\langle
\tilde{h}_{2}(\alpha w_{1}+\beta w_{2},\gamma w_{1}+\delta w_{2}),\ 
\frac{\partial h_{1}(\alpha w_{1}+\beta w_{2},\gamma w_{1}+\delta w_{2})}{\partial w_{1}}
\right\rangle
\leqslant
\tilde{m}_{2}
\binom{
8s+1}{
8s-1},
\]
and hence the second condition is attained. 

{\bf (iii)} For the last condition, in order for $\varphi$ to be a covering map with finite fibres, it suffices to find $(\alpha,\gamma)\in\mathbb{C}^{2}$ so that the holomorphic function of one variable
\[
h_{1}(\alpha w_{1}+\beta w_{2},
\gamma w_{1}+\delta w_{2})|_{\{w_{2}=0\}}
=
h_{1}(\alpha w_{1},\gamma w_{1})
\]
has a finite order of vanishing at $w_{1}=0$. To this effect, $h_{1}(z_{1},z_{2})$ may be written as an infinite sum 
\[
h_{1}(z_{1},z_{2})
=
P_{m_{1}}
+
\sum_{k\geqslant m_{1}+1}P_{k}
\]
of homogeneous polynomials $P_{k}$ of degree $k$, with $m_{1}=\mult_{0}\ h_{1}$. Hence,
\begin{eqnarray*}
h_{1}(\alpha w_{1},\gamma w_{1})
&=&
P_{m_{1}}(\alpha w_{1},\gamma w_{1})+O(w_{1}^{m_{1}+1})\\
&=&
w_{1}^{m_{1}}
P_{m}(\alpha,\gamma)+O(w_{1}^{m_{1}+1}).
\end{eqnarray*}
If $(\alpha,\gamma)\in \mathbb{C}^{2}-\{P_{m_{1}}=0\}$, then $h_{1}(\alpha w_{1},\gamma w_{1})\not\equiv 0$ and hence $\varphi$ defines a ramified $m_{1}$-sheeted analytic covering.

In summary, there exist $3\tilde{r}$ hyperplanes $H_{1}$,\dots,$H_{3\tilde{r}}$ in $\mathbb{C}^{\NN}$ so that for every 
\[
(c_{1},\dots,c_{\NN})\in 
\mathbb{C}^{\NN}
-
\bigcup_{k=1}^{3\tilde{r}}H_{k},
\]
there are $\tilde{r}$ hyperplanes $\tilde{H}_{1}$,\dots,$\tilde{H}_{\tilde{r}}$ and a hypersurface $\{P_{m_{1}}=0\}$ in $\mathbb{C}^{2}$ such that whenever 
\begin{eqnarray*}
\left(
\begin{matrix}
z_{1}\\
z_{2}
\end{matrix}
\right)
=
\left(
\begin{matrix}
\alpha & \beta\\
\gamma & \delta
\end{matrix}
\right)
\left(
\begin{matrix}
w_{1}\\
w_{2}
\end{matrix}
\right)
\end{eqnarray*}
is an invertible linear change of coordinate satisfying
\[
(\alpha,\gamma)\in 
\mathbb{C}^{2}
-
\bigcup_{k=1}^{\tilde{r}}\tilde{H}_{k}
-
\{P_{m_{1}}=0\},
\]
the three conditions {\bf (i)}, {\bf (ii)}, {\bf (iii)} are satisfied.
\end{proof}

\section{Explicit Calculation of $\varepsilon$ in Dimension $2$}

\subsubsection{}\label{cal-para-1} As before, we work in $\mathbb{C}^{2}$. Let $F_{1}$,\dots,$F_{\NN}$ be holomorphic function germs in $\mathcal{O}_{\mathbb{C}^{2},0}$ vanishing at the origin whose ideal $\mathcal{I}_{F}=\langle F_{1},\dots, F_{\NN}\rangle$ has finite intersection multiplicity with data $(p,q,s)$.

\subsubsection{}\label{cal-para-2} By \cite[p~1182]{Siu-2010}, one has for all $\phi\in\mathcal{D}_{0,1}(\Omega)$ with compact support that
\[
\||dF_{j}\cdot \phi|\|_{\frac{1}{4}}^{2}\lesssim Q(\phi,\phi)
\eqno
{\scriptstyle{(1\,\leqslant\, j\,\leqslant\, \NN).}}
\]

\subsubsection{}\label{cal-para-3} For any two vectors $(\lambda_{1},\dots,\lambda_{\NN})$, $(\mu_{1},\dots,\mu_{\NN})$ in $\mathbb{C}^{\NN}$, if 
\[
A=
\sum_{i=1}^{\NN}
\lambda_{i}F_{i}
\qquad
\text{and}
\qquad
B=
\sum_{i=1}^{\NN}
\mu_{i}F_{i},
\]
then by \ref{Kohn-Property}(iii)
\[
\||\Jac(A,B)\phi|\|_{\frac{1}{4}}^{2}
\lesssim
Q(\phi,\phi).
\]
\subsubsection{}\label{cal-para-4} By Corollary \ref{genselect-cor-main result}, there exist vectors $(\lambda_{1},\dots,\lambda_{\NN})$ and $(\mu_{1},\dots,\mu_{\NN})$ such that 
\[
\dim_{\mathbb{C}}\ 
\mathcal{O}_{\mathbb{C}^{2},0}
\big/
\langle A,B\rangle
\leqslant
4s^{2}.
\]
By Corollary \ref{LAG-cor-multiplicity of Jacobian is always lesser},
\[
\mult_{0}\ \Jac(A,B)
\leqslant 
4s^{2}-1.
\]

\subsubsection{}\label{cal-para-5} Write
\begin{eqnarray}\label{cal-eqn-jac}
\Jac(A,B)
=
f_{1}^{\alpha_{1}}
\cdots
f_{\tilde{r}}^{\alpha_{\tilde{r}}}
\end{eqnarray}
as a product of prime elements, and let $\alpha:=\max\{\alpha_{1},\dots,\alpha_{\tilde{r}}\}$. The holomorphic function
\[
\tilde{h}_{2}:=
f_{1}\cdots f_{\tilde{r}}
\]
is also a subelliptic multiplier since
\[
\tilde{h}_{2}^{\alpha}
=
f_{1}^{\alpha}\cdots f_{\tilde{r}}^{\alpha}
=
f_{1}^{\alpha-\alpha_{1}}
\cdots
f_{\tilde{r}}^{\alpha-\alpha_{\tilde{r}}}
\Jac(A,B)
\]
is a multiple of a subelliptic multiplier. Consequently, by radical property of subelliptic multipliers Proposition \ref{Kohn-Property}(i),
\[
\||\tilde{h}_{2}\phi|\|_{\frac{1}{4\alpha}}^{2}
\lesssim
Q(\phi,\phi).
\]
Moreover, by equation \eqref{cal-eqn-jac},
\[
\mult_{0}\ \Jac(A,B)
=
\sum_{i=1}^{\tilde{r}}\ 
\alpha_{i}\ \mult_{0}\ (f_{i})
\geqslant \alpha.
\]
Hence
\[
\frac{1}{4\alpha}
\geqslant
\frac{1}{4\ \mult_{0}\Jac(A,B)}
\geqslant 
\frac{1}{4(4s^{2}-1)},
\]
and 
\[
\||\tilde{h}_{2}\phi|\|_{\frac{1}{4(4s^{2}-1)}}^{2}
\lesssim
\||\tilde{h}_{2}\phi|\|_{\frac{1}{4\alpha}}^{2}
\lesssim
Q(\phi,\phi).
\]

\subsubsection{}\label{cal-para-6} As a remark,
\[
\mult_{0}\ \tilde{h}_{2}
=
\sum_{i=1}^{\tilde{r}}
\mult_{0}\ f_{i}
\leqslant
\sum_{i=1}^{\tilde{r}}
\alpha_{i}\mult_{0}\ f_{i}
=
\mult_{0}\ \Jac(A,B)
\leqslant
4s^{2}-1.
\]

\subsubsection{}\label{cal-para-7} By Proposition \ref{calprelim-prop-prelim for calculations-main results}, there exists $(c_{1},\dots,c_{\NN})\in\mathbb{C}^{\NN}$ and a linear change of coordinate $(z_{1},z_{2})\mapsto (w_{1},w_{2})$ via
\[
\left(
\begin{matrix}
z_{1}\\
z_{2}
\end{matrix}
\right)
=
\left(
\begin{matrix}
\alpha & \beta\\
\gamma & \delta
\end{matrix}
\right)
\left(
\begin{matrix}
w_{1}\\
w_{2}
\end{matrix}
\right)
\]
such that if $h_{1}=\sum_{k=1}^{\NN}c_{k}F_{k}$, one has 

\smallskip\noindent{\bf (i)}
\begin{eqnarray*}
&& \dim_{\mathbb{C}}\ 
\mathcal{O}_{\mathbb{C}^{2},0}
\big/
\langle h_{1}(\alpha w_{1}+\beta w_{2},\gamma w_{1}+\delta w_{2}),\tilde{h}_{2}(\alpha w_{1}+\beta w_{2},\gamma w_{1}+\delta w_{2})\rangle\\
&\leqslant & 
(\mult_{0}\ \tilde{h}_{2})s
\leqslant
(4s^{2}-1)s;
\end{eqnarray*}

\smallskip\noindent{\bf (ii)}
\begin{eqnarray*}
&& \dim_{\mathbb{C}}\ 
\mathcal{O}_{\mathbb{C}^{2},0}
\bigg/
\left\langle h_{1}(\alpha w_{1}+\beta w_{2},\gamma w_{1}+\delta w_{2}),\frac{\partial h_{1}(\alpha w_{1}+\beta w_{2},\gamma w_{1}+\delta w_{2})}{\partial w_{1}}\right\rangle\\
&\leqslant  &
(\mult_{0}\ \tilde{h}_{2})
{\binom{8s+1}{8s-1}}; 
\end{eqnarray*}

\smallskip\noindent{\bf (iii)}
if we let $(\mathbb{C}^{2},(w_{1},w_{1}))$ [resp. $(\mathbb{C}^{2},(x,y))$] denote $\mathbb{C}^{2}$ with coordinate system $(w_{1},w_{2})$ [resp. $(x,y)$], the holomorphic map 
\begin{eqnarray*}
\varphi:(\mathbb{C}^{2},(w_{1},w_{2})) &\longrightarrow &
(\mathbb{C}^{2},(x,y))\\
(w_{1},w_{2})
&\longmapsto &
(h_{1}(\alpha w_{1}+\beta w_{2},\gamma w_{1}+\delta w_{2}),w_{2})
\end{eqnarray*}
defines a ramified $\ord_{w_{1}=0}\ \tilde{h}_{1}(\alpha w_{1},\beta w_{1})$-cover over $\mathbb{C}^{2}$, which is therefore open and proper with finite fibres.

\subsubsection{}\label{cal-para-8} Since $h_{1}$ vanishes at the origin, by Lemma \ref{LAG-lem-if h(0)=0 then h^s is in the ideal},
\begin{eqnarray*}
&& h_{1}(\alpha w_{1}+\beta w_{2},\gamma w_{1}+\delta w_{2})^{
(4s^{2}-1)
{\binom{8s+1}{8s-1}}
}\\
&\in & 
\left\langle
\frac{\partial h_{1}(\alpha w_{1}+\beta w_{2},\gamma w_{1}+\delta w_{2})}{\partial z_{1}},
\tilde{h}_{2}(\alpha w_{1}+\beta w_{2},\gamma w_{1}+\delta w_{2})
\right\rangle.
\end{eqnarray*}

\subsubsection{}\label{cal-para-9} Let $C_{2}:=\{\tilde{h}_{2}(\alpha w_{1}+\beta w_{2},\gamma w_{1}+\delta w_{2})=0\}$ be the reduced curve. Since $\varphi$ is a proper map, by Remmert's proper mapping theorem, the image $\tilde{C}_{2}:=\varphi(C_{2})$ is an analytic set of dimension $1$.
 There exists an analytic function $h_{2}$ on $\mathbb{C}^{2}$ such that $\varphi(C_{2})=\{h_{2}(x,y)=0\}$. By Proposition \ref{propmap-prop-main result}, $\lambda:=\ord_{0}\ h_{2}(0,y)\leqslant (4s^{2}-1)s$. Hence, by Weierstrass' preparation theorem, there exist a unit $u(x,y)$, and holomorphic functions $a_{1}(x)$,\ \dots,\ $a_{(4s^{2}-1)s-1}(x)$ that vanish at $x=0$ such that $h_{2}(x,y)$ may be expressed as a Weierstrass polynomial
\[
h_{2}(x,y)
=
u(x,y)
\left(
y^{\lambda}
+
\sum_{k=0}^{\lambda-1}
a_{j}(x)y^{j}
\right)
\eqno
{\scriptstyle{\lambda\leqslant (4s^{2}-1)s}}.
\]

\subsubsection{}\label{cal-para-10} The holomorphic function $h_{2}(h_{1}(\alpha w_{1}+\beta w_{2},\gamma w_{1}+\delta w_{2}),w_{2})$ is also a subelliptic multiplier. More precisely, $h_{2}(h_{1}(\alpha w_{1}+\beta w_{2},\gamma w_{1}+\delta w_{2}),w_{2})$ is a multiple of $\tilde{h}_{2}(\alpha w_{1}+\beta w_{2},\gamma w_{1}+\delta w_{2})$ which is a subelliptic multiplier by paragraph  \ref{cal-para-5}. This follows from the fact (which will be explained below) that $V(\tilde{h}_{2})\subseteq V(h_{2}(h_{1},w_{2}))$ and hence by the  Nullstellensatz,
\[
\langle h_{2}(h_{1},w_{2})\rangle
\subseteq
\sqrt{\langle h_{2}(h_{1},w_{2})\rangle}
\subseteq
\sqrt{\langle \tilde{h}_{2}\rangle}
=
\langle \tilde{h}_{2}\rangle,
\]
where the equality follows from the fact that $\tilde{h}_{2}$ is reduced. 

Now to show that $V(\tilde{h}_{2}(\alpha w_{1}+\beta w_{2},\gamma w_{1}+\delta w_{2}))\subseteq V(h_{2}(h_{1}(\alpha w_{1}+\beta w_{2},\gamma w_{1}+\delta w_{2}),w_{2}))$, if $(\sigma,\mu)\in\mathbb{C}^{2}$ satisfies $\tilde{h}_{2}(\alpha \sigma+\beta \mu,\gamma \sigma+\delta \mu)=0$, then 
\[
\varphi(\sigma,\mu)
=(h_{1}(\alpha \sigma+\beta \mu,\gamma \sigma+\delta \mu),\mu)\in\{h_{2}(x,y)=0\}.
\]
Hence
\[
0=h_{2}(\varphi(\sigma,\mu))
=
h_{2}(h_{1}(\alpha \sigma+\beta \mu,\gamma \sigma+\delta \mu),\mu),
\]
from which we have proved the set inclusion. Consequently,
\[
\||h_{2}(h_{1}(\alpha w_{1}+\beta w_{2},\gamma w_{1}+\delta w_{2}),w_{2})\phi|\|_{\frac{1}{4(4s^{2}-1)}}^{2}
\lesssim
Q(\phi,\phi).
\]
Moreover, since $u(h_{1}(\alpha w_{1}+\beta w_{2},\gamma w_{1}+\delta w_{2}),w_{2})$ is a unit,
\[
\left\|\left|
\left(
w_{2}^{\lambda}
+
\sum_{j=1}^{\lambda-1}
a_{j}(h_{1}(\alpha w_{1}+\beta w_{2},\gamma w_{1}+\delta w_{2}))w_{2}^{j}
\right)
\phi
\right|\right\|_{\frac{1}{4(4s^{2}-1)}}^{2}
\lesssim
Q(\phi,\phi),
\]
where $\lambda\leqslant (4s^{2}-1)s$.

\subsubsection{}\label{cal-para-11}
To declutter notations, we will set
\begin{eqnarray*}
h_{1}(w_{1},w_{2}) &:=& h_{1}(\alpha w_{1}+\beta w_{2},\gamma w_{1}+\delta w_{2})\\
\frac{\partial h_{1}(w_{1},w_{2})}{\partial w_{1}}
&:=&
\frac{\partial h_{1}(\alpha w_{1}+\beta w_{2},\gamma w_{1}+\delta w_{2})}{\partial w_{1}}\\
\tilde{h}_{2}(w_{1},w_{2})
&:=&
\tilde{h}_{2}(\alpha w_{1}+\beta w_{2},\gamma w_{1}+\delta w_{2})\\
\eta &:=& (4s^{2}-1){\binom{8s+1}{8s-1}}\\
\lambda &:=& \ord_{0}\ h_{2}(0,y)\leqslant (4s^{2}-1)s.
\end{eqnarray*}
By Paragraph \ref{cal-para-8}, since
\[
h_{1}^{\eta}
\in 
\left\langle
\frac{\partial h_{1}}{\partial w_{1}},\ 
\tilde{h}_{2}
\right\rangle,
\]
there is an estimate
\[
|h_{1}^{\eta}|
\lesssim
\left|\frac{\partial h_{1}}{\partial w_{1}}
\right|
+
|\tilde{h}_{2}|.
\]

\subsection{Siu's method: Starting Point}
\subsubsection{} Since $h_{1}$ is a pre-multiplier and 
\[
dh_{1}\wedge dh_{2}
=
\frac{\partial h_{1}}{\partial w_{1}}
\left(
\lambda w_{2}^{\lambda-1}
+
\sum_{j=1}^{\lambda-1}
ja_{j}(h_{1})w_{2}^{j-1}
\right)
dw_{1}\wedge dw_{2},
\]
the holomorphic function 
\[
\frac{\partial h_{1}}{\partial w_{1}}
\left(
\lambda w_{2}^{\lambda-1}
+
\sum_{j=1}^{\lambda-1}
ja_{j}(h_{1})w_{2}^{j-1}
\right)
\]
is also a subelliptic multiplier and we will estimate its regularity property. Since 
\[
\left\|\left|
\left(w_{2}^{\lambda}
+
\sum_{j=0}^{\lambda-1}
a_{j}(h_{1})w_{2}^{j}
\right)\phi
\right|\right\|_{\frac{1}{4(4s^{2}-1)}}^{2}
\lesssim
Q(\phi,\phi),
\]
using Proposition \ref{Kohn-Property}(ii), 
\[
\left\|\left|
d\left(
w_{2}^{\lambda}
+
\sum_{j=0}^{\lambda-1}
a_{j}(h_{1})w_{2}^{j}
\right)\cdot \phi
\right|\right\|_{\frac{1}{8(4s^{2}-1)}}^{2}
\lesssim
Q(\phi,\phi).
\]
Also,
\[
\||dh_{1}\cdot \phi|\|_{\frac{1}{8(4s^{2}-1)}}^{2}
\lesssim
\||dh_{1}\cdot \phi|\|_{\frac{1}{4}}^{2}
\lesssim
Q(\phi,\phi),
\]
where the last inequality comes from Paragraph \ref{cal-para-2} and the fact that $h_{1}$ is a linear combination of the $F_{i}$. 
By Proposition \ref{Kohn-Property}(iii), the regularity of the subelliptic multiplier is obtained below
\[
\left\|\left|
\frac{\partial h_{1}}{\partial w_{1}}
\left(
\lambda w_{2}^{\lambda-1}
+
\sum_{j=1}^{\lambda-1}
ja_{j}(h_{1})w_{2}^{j-1}
\right)
\phi
\right|\right\|_{\frac{1}{8(4s^{2}-1)}}^{2}
\lesssim
Q(\phi,\phi).
\]

\subsubsection{} Since $\tilde{h}_{2}$ is also a subelliptic multiplier, so is
\[
\tilde{h}_{2}\left(
\lambda w_{2}^{\lambda-1}
+
\sum_{j=1}^{\lambda-1}
ja_{j}(h_{1})w_{2}^{j-1}
\right).
\]
Hence by paragraph \ref{cal-para-5},
\[
\left\|\left|
\tilde{h}_{2}\left(
\lambda w_{2}^{\lambda-1}
+
\sum_{j=1}^{\lambda-1}
ja_{j}(h_{1})w_{2}^{j-1}
\right)
\phi
\right\|\right|_{\frac{1}{4(4s^{2}-1)}}^{2}
\lesssim
Q(\phi,\phi).
\]

\subsubsection{} By the previous two paragraphs and the inequality in \ref{cal-para-11},
\[
|h_{1}^{\eta}|
\lesssim
\left|
\frac{\partial h_{1}}{\partial z_{1}}
\right|
+
|\tilde{h}_{2}|,
\]
there is an estimate
\begin{eqnarray*}
&& \left\|\left|
h_{1}^{\eta}\left(
\lambda w_{2}^{\lambda-1}
+
\sum_{j=1}^{\lambda-1}
ja_{j}(h_{1})w_{2}^{j-1}
\right)\phi
\right|\right\|_{\frac{1}{8(4s^{2}-1)}}^{2}\\
&\lesssim &
\left\|\left|
\frac{\partial h_{1}}{\partial w_{1}}
\left(
\lambda w_{2}^{\lambda-1}
+
\sum_{j=1}^{\lambda-1}
ja_{j}(h_{1})w_{2}^{j-1}
\right)\phi
\right|\right\|_{\frac{1}{8(4s^{2}-1)}}^{2}
+ \left\|\left|
\tilde{h}_{2}\left(
\lambda w_{2}^{\lambda-1}
+
\sum_{j=1}^{\lambda-1}
ja_{j}(h_{1})w_{2}^{j-1}
\right)\phi
\right|\right\|_{\frac{1}{8(4s^{2}-1)}}^{2}\\
&\lesssim & Q(\phi,\phi).
\end{eqnarray*}

\subsection{Siu's method: Inductive Step} 

\subsubsection{}
Let $h_{2}^{(0)}:=h_{2}$ and 
\[
h_{2}^{(1)}
:=
h_{1}^{\eta}\left(
\lambda w_{2}^{\lambda-1}
+
\sum_{j=1}^{\lambda-1}
ja_{j}(h_{1})w_{2}^{j-1}
\right).
\]
For $1\leqslant \nu\leqslant \lambda$, define
\[
h_{2}^{(\nu)}
:=
h_{1}^{\nu\eta}
\left(
\frac{\lambda!}{(\lambda-\nu)!}
w_{2}^{\lambda-\nu}
+
\sum_{j=\nu}^{\lambda-1}
\frac{j!}{(j-\nu)!}a_{j}(h_{1})
w_{2}^{j-\nu}
\right),
\]
which will be shown that it is also a subelliptic multiplier and 
\[
\||h_{2}^{(\nu)}
\phi|\|_{\frac{1}{2^{\nu}\cdot 4(4s^{2}-1)}}^{2}
\lesssim
Q(\phi,\phi).
\]

\subsubsection{} We will first calculate something analogous to the first paragraph of the previous subsection. Suppose that the induction statement is true for $\nu-1$, meaning that 
\[
\||h_{2}^{(\nu-1)}\phi|\|_{\frac{1}{2^{\nu-1}\cdot 4(4s^{2}-1)}}^{2}
\lesssim 
Q(\phi,\phi).
\]
Then 
\[
\||
dh_{2}^{(\nu-1)}\phi
|\|_{\frac{1}{2^{\nu}\cdot 4(4s^{2}-1)}}^{2}
\lesssim
Q(\phi,\phi).
\]
Moreover,
\[
\||dh_{1}\cdot \phi|\|_{\frac{1}{4}}^{2}\lesssim Q(\phi,\phi).
\]
Therefore,
\[
dh_{1}\wedge dh_{2}^{(\nu-1)}
=
\frac{\partial h_{1}}{\partial w_{1}}
\left(
h_{1}^{\eta(\nu-1)}
\left(
\frac{\lambda!}{(\lambda-\nu)!}
w_{2}^{\lambda-\nu}
+
\sum_{j=\nu}^{\lambda-1}
\frac{j!}{(j-\nu)!}
a_{j}(h_{1})w_{2}^{j-\nu}
\right)
\right)
dw_{1}\wedge dw_{2},
\]
whose coefficient is also a subelliptic multiplier with 
\[
\left\|\left|
\frac{\partial h_{1}}{\partial w_{1}}
\left(
h_{1}^{\eta(\nu-1)}
\left(
\frac{\lambda!}{(\lambda-\nu)!}
w_{2}^{\lambda-\nu}
+
\sum_{j=\nu}^{\lambda-1}
\frac{j!}{(j-\nu)!}
a_{j}(h_{1})w_{2}^{j-\nu}
\right)
\right)\phi
\right|\right\|_{\frac{1}{2^{\nu}\cdot 4(4s^{2}-1)}}^{2}
\lesssim
Q(\phi,\phi).
\]

\subsubsection{} Since $\tilde{h}_{2}$ is also a subelliptic multiplier, 
\[
\left\|\left|
\tilde{h}_{2}
\left(
h_{1}^{\eta(\nu-1)}
\left(
\frac{\lambda!}{(\lambda-\nu)!}
w_{2}^{\lambda-\nu}
+
\sum_{j=\nu}^{\lambda-1}
\frac{j!}{(j-\nu)!}
a_{j}(h_{1})w_{2}^{j-\nu}
\right)
\right)\phi
\right|\right\|_{\frac{1}{2^{\nu}\cdot 4(4s^{2}-1)}}^{2}
\lesssim
Q(\phi,\phi).
\]

\subsubsection{} Combining the inequalities in the last two paragraphs, and using the fact that $|h_{1}^{\eta}|\lesssim |\partial_{z_{1}}h_{1}|+|\tilde{h}_{2}|$,
\begin{eqnarray*}
&& \||h_{2}^{(\nu)}\phi|\|_{\frac{1}{2^{\nu}\cdot 4(4s^{2}-1)}}^{2}\\
&=& \left\|\left|
h_{1}^{\nu\eta}
\left(
\frac{\lambda!}{(\lambda-\nu)!}
w_{2}^{\lambda-\nu}
+
\sum_{j=\nu}^{\lambda-1}
\frac{j!}{(j-\nu)!}
a_{j}(h_{1})w_{2}^{j-\nu}
\right)\phi
\right|\right\|_{\frac{1}{2^{\nu}\cdot 4(4s^{2}-1)}}^{2}\\
&=& \left\|\left|
h_{1}^{\eta}\left(h_{1}^{\eta(\nu-1)}
\left(
\frac{\lambda!}{(\lambda-\nu)!}
w_{2}^{\lambda-\nu}
+
\sum_{j=\nu}^{\lambda-1}
\frac{j!}{(j-\nu)!}
a_{j}(h_{1})w_{2}^{j-\nu}
\right)\right)\phi
\right|\right\|_{\frac{1}{2^{\nu}\cdot 4(4s^{2}-1)}}^{2}\\
&\lesssim & 
\left\|\left|
\frac{\partial h_{1}}{\partial w_{1}}
\left(h_{1}^{\eta(\nu-1)}
\left(
\frac{\lambda!}{(\lambda-\nu)!}
w_{2}^{\lambda-\nu}
+
\sum_{j=\nu}^{\lambda-1}
\frac{j!}{(j-\nu)!}
a_{j}(h_{1})w_{2}^{j-\nu}
\right)\right)\phi
\right|\right\|_{\frac{1}{2^{\nu}\cdot 4(4s^{2}-1)}}^{2}\\
&& +
\left\|\left|
\tilde{h}_{2}\left(h_{1}^{\eta(\nu-1)}
\left(
\frac{\lambda!}{(\lambda-\nu)!}
w_{2}^{\lambda-\nu}
+
\sum_{j=\nu}^{\lambda-1}
\frac{j!}{(j-\nu)!}
a_{j}(h_{1})w_{2}^{j-\nu}
\right)\right)\phi
\right|\right\|_{\frac{1}{2^{\nu}\cdot 4(4s^{2}-1)}}^{2}\\
&\lesssim & 
Q(\phi,\phi).
\end{eqnarray*}
This finishes the induction process.

\subsubsection{} Setting $\nu=\lambda$, we get
\[
\||h_{2}^{(\lambda)}\phi|\|_{\frac{1}{2^{\lambda}\cdot 4(4s^{2}-1)}}^{2}
\lesssim 
Q(\phi,\phi).
\]
But $h_{2}^{(\lambda)}=h_{1}^{\eta\lambda}
\lambda!$, therefore
\[
\||h_{1}^{\lambda\eta}\phi
|\|_{\frac{1}{2^{\lambda}\cdot 4(4s^{2}-1)}}^{2}
\lesssim
Q(\phi,\phi).
\]
\subsection{Siu's method: Conclusion and End of Calculation}
\subsubsection{} Since 
\[
\dim_{\mathbb{C}}\ 
\mathcal{O}_{\mathbb{C}^{2},0}
\big/
\langle h_{1},\tilde{h}_{2}\rangle
\leqslant (4s^{2}-1)s,
\]
by Proposition \ref{LAG-prop-intersectionnumberofproductfg=sumofintersection},
\[
\dim_{\mathbb{C}}\ 
\mathcal{O}_{\mathbb{C}^{2},0}
\big/
\langle h_{1}^{\lambda\eta},\tilde{h}_{2}\rangle
=
\eta\lambda\ 
\dim_{\mathbb{C}}\ 
\mathcal{O}_{\mathbb{C}^{2},0}
\big/
\langle h_{1},\tilde{h}_{2}\rangle
\leqslant
(4s^{2}-1)s\eta\lambda.
\]
For $i=1,\ 2$, by the  Lemma \ref{LAG-lem-if h(0)=0 then h^s is in the ideal},
\[
w_{i}^{(4s^{2}-1)s\lambda\eta}
\in
\left\langle
h_{1}^{\lambda\eta},
\tilde{h}_{2}
\right\rangle.
\]
Thus $w_{i}^{(4s^{2}-1)s\lambda\eta}$ is also a multiplier with 
\[
|w_{i}^{(4s^{2}-1)s\lambda\eta}|
\lesssim 
|h_{1}^{\eta\lambda}|
+
|\tilde{h}_{2}|.
\]
Hence
\begin{eqnarray*}
&& \||w_{i}^{(4s^{2}-1)s\lambda\eta}\phi|\|_{\frac{1}{2^{\lambda}\cdot 4(4s^{2}-1)}}^{2}\\
&\lesssim & 
\||h_{1}^{\eta\lambda}\phi|\|_{\frac{1}{2^{\lambda}\cdot 4(4s^{2}-1)}}^{2}
+
\||\tilde{h}_{2}\phi|\|_{\frac{1}{2^{\lambda}\cdot 4(4s^{2}-1)}}^{2}\\
&\lesssim & 
\||h_{1}^{\eta\lambda}\phi|\|_{\frac{1}{2^{\lambda}\cdot 4(4s^{2}-1)}}^{2}
+
\||\tilde{h}_{2}\phi|\|_{\frac{1}{4(4s^{2}-1)}}^{2}
\lesssim Q(\phi,\phi).
\end{eqnarray*}
By radical property of subelliptic multipliers Proposition \ref{Kohn-Property}(i), one has for each $i=1,\ 2$ that
\[
\||w_{i}\phi|\|_{\frac{1}{2^{\lambda}\cdot 4s\eta\lambda(4s^{2}-1)^{2}}}^{2}
\lesssim 
Q(\phi,\phi).
\]
Taking the Jacobian, one obtains by Propositions \ref{Kohn-Property}(ii) and \ref{Kohn-Property}(iii) that 
\[
\||\phi|\|_{\frac{1}{2^{\lambda+1}\cdot 4s\eta\lambda(4s^{2}-1)^{2}}}^{2}
\lesssim 
Q(\phi,\phi),
\]
and this terminates the calculation.

\nocite{*}

\bibliographystyle{alpha}
\bibliography{ee-version-3.bib}


\end{document}